\documentclass[11pt]{amsart}
\usepackage{amssymb}
\usepackage{amsmath}
\newtheorem{theorem}{Theorem}[section]
\newtheorem{lemma}[theorem]{Lemma}
\newtheorem{remark}[theorem]{Remark}
\newtheorem{definition}[theorem]{Definition}
\newtheorem{example}[theorem]{Example}
\newtheorem{proposition}[theorem]{Proposition}
\newtheorem{corollary}[theorem]{Corollary}

\def\u{\mathfrak{A}}
 \def\s{\{\u_k\}_k}
\def\hbu{H_{b\u}}
\def\hu{H_{\u}}
\def\mbu{M_{b\u}}

\def\p{\mathcal{P}}
\def\v{\overset\vee}
 \def\l{\mathcal{L}}

 \def\bs{\backslash}
 \def\hdu{H_{d\u}}
 \newcommand{\TS[2]}{\bigotimes\nolimits^{#2,s}} %producto tensorial sim\'{e}trico
\begin{document}

\title{On algebras of holomorphic functions of a given type}

\author{Santiago Muro}

% \thanks{Partially supported by ANPCyT PICT 05 17-33042. The first and third authors were also partially supported by UBACyT Grant X038 and ANPCyT PICT 06 00587.}

\address{Departamento de Matem\'{a}tica - Pab I,
Facultad de Cs. Exactas y Naturales, Universidad de Buenos Aires,
(1428) Buenos Aires, Argentina and CONICET} \email{smuro@dm.uba.ar}

% \subjclass[2000]{47H60, 46G20, 30H05, 46M05} 

\begin{abstract}
We show that several spaces of holomorphic functions on a Riemann domain over a Banach space, including the nuclear and Hilbert-Schmidt bounded type, are locally $m$-convex Fr\'echet algebras. 
We prove that the spectrum of these algebras has a natural analytic structure, which we use to characterize the envelope of holomorphy. We also show a Cartan-Thullen type theorem.% As a by-product, we improve a previously known bound for the norm of the product of Hilbert-Schmidt polynomials.
\end{abstract}

 \keywords{holomorphy types, polynomial ideals, Fr\'echet algebras, Riemann domains}

\maketitle

\section*{Introduction}

Holomorphy types were defined by Nachbin \cite{Nac69} in order to include in a single theory the most commonly used spaces of holomorphic functions on infinite dimensional spaces, like the current, nuclear, compact or Hilbert-Schmidt type.
Since Pietsch' definition of ideals of multilinear forms \cite{Pie84}, the theory of holomorphy types began to interact with the concept of normed ideal of homogeneous polynomials (see for example \cite{Hol86,BraJun90,BotBraJunPel06,CarDimMur09}). We follow this approach in this article incorporating them to the definition of holomorphy type.

A holomorphy type $\u$ at a Banach space $E$ is a sequence $\{\u_n(E)\}$ of normed spaces of $n$-homogeneous polynomials with the property that if a polynomial is in $\u$, then all its differentials also belong to $\u$ and such that their norms have controlled growth. 
In \cite{Har97}, Harris deduced a tight bound for the usual norm of the differentials of a homogeneous polynomial. 
Noticeably, we are able to show that every known (to us) example of
holomorphy type satisfies the same bound (see the examples of the first section).

A holomorphic function on an open set $U$ of $E$ is of type $\u$ if
it has positive $\u$-radius of convergence at each point of $U$
\cite{Nac69, Din71(holomorphy-types), Aro73}. Similarly, entire
functions of bounded $\u$-type are defined as functions that have
infinite $\u$-radius of convergence at zero
\cite{FavJat09,CarDimMur07}, and  it
is immediate to generalize this definition to holomorphic functions
of bounded $\u$-type on a ball. We propose a definition of holomorphic function of bounded $\u$-type on a general open set  $U$ of $E$ and on a Riemann domain spread over $E$, and study some properties of the space $\hbu(U)$ of this class of functions.
When $\u$ is the sequence of all continuous homogeneous polynomials (i.e. for the current holomorphy type) we recover  $H_b(U)$, the space of bounded type holomorphic functions on $U$. The space $H_b(U)$ is a
Fr\'echet algebra with the topology of uniform convergence on $U$-bounded sets. On the other hand, the spaces of nuclear
\cite{Din71(holomorphy-types)} (see also
\cite[Exercise 2.63]{Din99}) and Hilbert-Schmidt \cite{Pet01}
polynomials were proved to be an algebra, and it is easy to deduce that the corresponding spaces of entire
functions of bounded type are also algebras. A more general approach was followed in
\cite{CarDimMur07,CarDimMur}, where \emph{multiplicative} sequences of normed ideals of polynomials were studied. A sequence $\u$ of normed ideals of polynomials is multiplicative if a bound for the norm of the product of two homogeneous polynomials in $\u$ can be obtained in terms of the product of the norms of these polynomials. We showed there
that several spaces of  entire functions of bounded type are algebras
(with continuous multiplication). In Section \ref{section multiplicative}, we prove 
that for every previously mentioned example of holomorphy
type the corresponding space $\hbu(U)$, is actually a locally $m$-convex Fr\'echet algebra with its natural
topology. We also show that algebras $\hbu(E)$ are test algebras for the famous Michael's problem on the continuity of characters.

We study the spectrum of the algebra $\hbu(U)$ and show in Theorem \ref{mbu dom riemann} that, under fairly general conditions on $\u$, it may be endowed with a structure of  Riemann domain over the bidual of $E$, generalizing some of the results in \cite{AroGalGarMae96,CarDimMur}. We show in Section \ref{section holo extensions} that the Gelfand extensions to the spectrum are holomorphic and that the spectrum is a domain of holomorphy with respect to the set of all holomorphic functions, as it was proved in \cite{DinVen04} for the current holomorphy type. 
We also characterize  the $\hbu$-envelope of holomorphy of an open set $U$ as a part of the spectrum of $\hbu(U)$.
The envelope of holomorphy for the space of holomorphic functions of a given type was constructed by Hirschowitz \cite{Hir72}.  He also raised the question whether the holomorphic extensions to the envelope of holomorphy are of the same type, see \cite[p. 289-290]{Hir72}. We investigate this question for the case of holomorphic functions of bounded $\u$-type in the last section of the article. We need to deal with \emph{weakly differentiable} sequences. 
The concept of weak differentiability was  defined in \cite{CarDimMur} and it was proved there that it is, in some sense, dual to multiplicativity (see also Remark \ref{w-dif dual a mult}). The importance of this duality can be seen, for instance, in Example \ref{HS multip}, where we prove that the Hilbert-Schmidt norm of the product of two homogeneous polynomials is bounded by the product of their Hilbert-Schmidt norms, improving a result obtained by Petersson \cite{Pet01}.
When the sequence $\u$ is  weakly differentiable we succeed to show that the extension of a function in $\hbu(U)$ to the $\hbu$-envelope of holomorphy of $U$ is of type $\u$, thus answering in this case the question of Hirschowitz positively. On the other hand, it is known (see \cite[Example 2.8]{CarMur}) that the extension of a bounded type holomorphic function to the $H_b$-envelope of holomorphy may fail to be of bounded type, thus, we cannot expect the extension of every function in $\hbu(U)$ to be of bounded $\u$-type on the $\hbu$-envelope. 
We end the article with a version of the Cartan-Thullen theorem for $\hbu(U)$.

% \textcolor{red}{la s de T converge en balanceados?}

 We refer to \cite{Din99,Muj86} for notation and results regarding
polynomials and holomorphic functions in general and to
\cite{DefFlo93,Flo01,Flo02,FloHun02} for polynomial ideals and tensor
products of Banach spaces.

\section{Preliminaries}\label{section prelim}

Throughout this article $E$ is a complex Banach space and $B_E(x,r)$ denotes the open ball of radius $r$ and center $x$ in $E$. 
We denote by $\p^k(E)$ the Banach space of all continuous
$k$-homogeneous polynomials from $E$ to $\mathbb{C}$.

We define, for each $P\in\p^k(E)$, $a\in E$ and $j\le k$ the polynomial $P_{a^j}\in \p^{k-j}(E)$ by $$
P_{a^j}(x)=\v P(a^j,x^{k-j})=\v
P(\underbrace{a,...,a}_j,\underbrace{x,...,x}_{k-j}),$$ where $\v P$ is the symmetric $k$-linear form associated to $P$. For $j=1$,
we write $P_a$ instead of $P_{a^1}$.

A Riemann domain spread over $E$ is a pair $(X,q)$, where $X$ is a Hausdorff topological space and $q:X\to E$ is a local homeomorphism. For $x\in X$, a ball of radius $r>0$ centered at $x$ is a neighbourhood of $x$ that is homeomorphic to $B_E(q(x),r)$ through $q$. It will be denoted by $B_X(x,r)$. When there is no place for confusion we denote the ball of center $x$ and radius $r$ by $B_r(x)$ (where $x$ can be in $E$ or in $X$). We also define the distance to the border of $X$, which is a function $d_X:X\to \mathbb R_{>0}$ defined by $d_X(x)=\sup\{r>0:\,B_r(x)\textrm{ exists}\}$. For a subset $A$ of $X$, $d_X(A)$ is defined as the infimum of $d_X(x)$ with $x$ in $A$. A subset $A$ of $X$ is called an $X$-bounded subset if $d_X(A)>0$ and $q(A)$ is a bounded set in $E$.

\bigskip

Let us recall the definition of polynomial ideal
\cite{Flo01,Flo02}. 
\begin{definition}\rm
A \textbf{Banach ideal of
(scalar-valued) continuous $k$-homo\-geneous polynomials} is a pair
$(\mathfrak{A}_k,\|\cdot\|_{\mathfrak A_k})$ such that:
\begin{enumerate}
\item[(i)] For every Banach space $E$, $\mathfrak{A}_k(E)=\mathfrak A_k\cap
\mathcal
P^k(E)$ is a linear subspace of $\p^k(E)$ and $\|\cdot\|_{\u_k(E)}$
is a norm on it. Moreover, $(\u_k(E), \|\cdot\|_{\u_k(E)})$ is a
Banach space.

\item[(ii)] If $T\in \l (E_1,E)$ and $P \in \u_k(E)$, then $P\circ T\in
\u_k(E_1)$ with $$ \|
P\circ T\|_{\u_k(E_1)}\le \|P\|_{\u_k(E)} \| T\|^k.$$

\item[(iii)] $z\mapsto z^k$ belongs to $\u_k(\mathbb C)$
and has norm 1.
\end{enumerate} 
\end{definition}

We use the following version of the concept of holomorphy type.
\begin{definition}\label{defihtype}\rm
Consider the sequence $\u=\{\u_k\}_{k=1}^\infty$, where for each
$k$, $\u_k$ is a Banach ideal of  $k$-homogeneous
polynomials. We say that $\{\u_k\}_k$ is a \textbf{holomorphy type} if for each $l<k$ there exist a positive constant $c_{k,l}$
such that for every Banach
space $E$, the following hold:
\begin{equation}\label{htype}
\textrm{if }P\in \u_{k}(E)\textrm{, }a\in E\textrm{ then } P_{a^l}\textrm{ belongs to
}\u_{k-l}(E)\textrm{ and }
\|P_{a^l}\|_{\u_{k-l}(E)} \le c_{k,l}
\|P\|_{\u_{k}(E)} \|a\|^l.
\end{equation}
\end{definition}
\begin{remark}\rm
(a)  The difference between the above definition and the original Nachbin's definition of holomorphy type \cite{Nac69} is twofold. First, Nachbin did not work with polynomial ideals, a concept that was not defined until mid 80's after the work of Pietsch \cite{Pie84}. We think however that polynomial ideals are in the spirit of the concept of holomorphy type. Holomorphy types defined as above are global holomorphy types in the sense of \cite{BotBraJunPel06}.
Second, the constants considered by Nachbin were of the form $c_{k,l}=\binom{k}{l}C^k$ for some fixed constant $C$.
In most of the results we require that the constants satisfy, for every $k,l$,
\begin{equation}\label{constantes}
c_{k,l} \le\frac{(k+l)^{k+l}}{(k+l)!}\frac{k!}{k^k}\frac{l!}{l^l}.
\end{equation}
These constants are more restrictive than Nachbin's constants, but, as we will see bellow, the constants $c_{k,l}$ of every usual example of holomorphy type satisfy (\ref{constantes}).
\end{remark}

\begin{remark}\label{cohe dual}\rm
In \cite{CarDimMur09} we defined and studied \textbf{coherent sequences} of
polynomial ideals.
Any coherent sequence is a holomorphy type. In fact, a coherent sequence is a holomorphy type which satisfies the following extra condition: for each $l,k\in\mathbb N$ there exists a positive constant $d_{k,l}$
such that for every Banach
space $E$,
\begin{equation}\label{dual a fijar}
\textrm{if }P\in \u_{k}(E)\textrm{, }\gamma\in E'\textrm{ then } \gamma^l P\textrm{ belongs to
}\u_{k+l}(E)\textrm{ and }
\|\gamma^lP\|_{\u_{k+l}(E)} \le d_{k,l}
\|P\|_{\u_{k}(E)} \|\gamma\|^l.
\end{equation}
The constants appearing in \cite{CarDimMur09} were of the form $c_{k,l}=C^l$ and $d_{k,l}=D^l$.
The extra condition asked for the coherence is very natural since conditions (\ref{htype}) and (\ref{dual a fijar}) are dual to each other in the following sense: if $\{\u_k\}_{k=1}^\infty$ is a sequence of polynomial ideals which satisfies (\ref{htype}) (respectively (\ref{dual a fijar})) then the sequence of adjoint ideals $\{\u_k^*\}_{k=1}^\infty$ satisfies (\ref{dual a fijar}) (respectively (\ref{htype})) with the same constants (see \cite[Proposition 5.1]{CarDimMur09}). Thus we may think a coherent sequence as a sequence of polynomial ideals which form a holomorphy type and whose adjoints ideals are also a holomorphy type.
\end{remark}

We present now some examples of holomorphy types with constants as in (\ref{constantes}).

\begin{example}
The sequence $\p$ of ideals of continuous polynomials \rm is, by \cite[Corollary 4]{Har97}, a holomorphy type  with constants as in (\ref{constantes}). The same holds for other sequences of closed ideals as the sequence $\p_{w}$ of weakly continuous on bounded sets polynomials, and the sequence $\p_A$ of approximable polynomials.
\end{example}
Slight modifications on the results of \cite[Corollaries 5.2 and 5.6]{CarDimMur09} (see also \cite[Section 3.1.4]{Mur10}) show that if a sequence $\{\u_k\}_{k=1}^\infty$ of ideals form a holomorphy type with constants $c_{k,l}$ the sequence of maximal ideals $\{\u_k^{\max}\}_{k=1}^\infty$ and the sequence of minimal ideals $\{\u_k^{\min}\}_{k=1}^\infty$ are also holomorphy types with the same constants $c_{k,l}$. Also, as already mentioned in Remark \ref{cohe dual}, if the sequence $\{\u_k\}_{k=1}^\infty$ satisfies the condition of coherence (\ref{dual a fijar}) with constants $d_{k,l}$ then the sequence of adjoint ideals  $\{\u_k^*\}_{k=1}^\infty$  is a holomorphy type with constants $d_{k,l}$. As a consequence we have the following.
\begin{example}
The sequence $\p_I$ of ideals of integral polynomials \rm  is a holomorphy type with constants $c_{k,l}=1$. Indeed, since condition (\ref{dual a fijar})
is trivially satisfied by the sequence $\p$ with constants $d_{k,l}=1$ and since $(\p^k)^*=\p_I^k$, the result follows from the above comments.
\end{example}
\begin{example}
The sequence $\p_N$ of ideals of nuclear polynomials \rm is a holomorphy type with constants $c_{k,l}=1$ because $\p^k_N=(\p^k_I)^{\min}$.
\end{example}
\begin{example} Sequences of polynomial ideals associated to a sequence of natural symmetric tensor norms. \rm
For a symmetric tensor norm $\beta_k$ (of order $k$), the projective
and injective hulls of $\beta_k$ (denoted
 as $\bs \beta_k /$ and $/
\beta_k \bs$ respectively) are defined as the tensor norms
induced by the following mappings (see \cite{CarGal-4}):
$$ \big( \otimes^{k,s} \ell_1(B_E),  \beta_k  \big) \overset 1 \twoheadrightarrow \big( \otimes^{k,s}  E,   \bs \beta_k /  \big).$$
$$ \big( \otimes^{k,s} E, / \beta_k \bs \big) \overset 1 \hookrightarrow  \big( \otimes^{k,s} \ell_{\infty}(B_{{E}'}),  \beta_k \big).$$

In \cite{CarGal-4}, natural tensor norms for arbitrary order are
introduced and studied, in the spirit of the natural tensor norms of
Grothendieck.
A finitely generated symmetric tensor norm of order
$k$, $\beta_k$ is \textit{natural} if $\beta_k$ is obtained from $\varepsilon_k$ (the injective symmetric tensor norm) and $\pi_k$ (the projective symmetric tensor norms) taking a finite number of projective and injective hulls (see \cite{CarGal-4} for details).
 For $k\ge 3$, it is
shown in \cite{CarGal-4} that there are exactly six non-equivalent
natural tensor norms and for $k=2$ there are only four.

Let $\u_k$ be an ideal of $k$-homogeneous polynomials associated to a finitely generated symmetric tensor norm $\alpha_k$. Small variations in Lemma 3.1.34 of  \cite{Mur10} show that if $\{\u_k\}$ is a holomorphy type with constants $c_{k,l}$ then so are the sequences of maximal  (or minimal) ideals associated to the projective and injective hulls of $\alpha_k$. In particular, any of the sequences of maximal  (or minimal) ideals associated to any of the sequences of natural norms is a holomorphy type with constants as in (\ref{constantes}). 
\end{example}

\begin{example}The sequence $\p_e$ of ideals of extendible polynomials. \rm
 Since the ideal of extendible polynomials $\p_e^k$ is the maximal ideal associated to the tensor norm $\bs\varepsilon_k/$, we have by the previous example that the sequence $\p_e$ is a holomorphy type with constants as in (\ref{constantes}).
\end{example}

\begin{example} The sequence $\mathcal M_r$ of ideals of multiple $r$-summing polynomials. \rm
  It was shown in \cite[Example 1.13]{CarDimMur09} that it is a coherent sequence with constants equal to 1 thus, in particular, it is a holomorphy type with constants $c_{k,l}=1$. 
\end{example}

\begin{example} The sequence $\mathcal S_2$ of ideals of Hilbert-Schmidt polynomials. \rm
It was shown in \cite[Proposition 3]{Dwy71} that it is a holomorphy type with constants $c_{k,l}=1$.
\end{example}
\begin{example} The sequence $\mathcal S_p$ of $p$-Schatten-von Neumann polynomials. \rm
Let $H$ be a Hilbert space. Recall that for $1< p<\infty$, the $p$-Schatten-von Neumann $k$-homogeneous polynomials on $H$ may be defined, using the complex interpolation method,  interpolating  nuclear and approximable polynomials on $H$ \cite{CobKuhPee92,CarDimMur07} as follows:
$$
\mathcal S^k_p(H):=[\p^k_N(H),\p^k_A(H)]_{\theta},
$$
where $p(1-\theta)=1$. The space of Hilbert-Schmidt polynomials coincide (isometrically) with the space of $2$-Schatten-von Neumann polynomials.
Since interpolation of holomorphy types is a holomorphy type (see \cite[Proposition 1.2]{CarDimMur}), we can conclude that $\{\mathcal S_p^k\}$ is a holomorphy type with constants $$c_{k,l} \le\Big(\frac{(k+l)^{k+l}}{(k+l)!}\frac{k!}{k^k}\frac{l!}{l^l}\Big)^{1-\frac1{p}}. $$
Moreover, using the Reiteration Theorem \cite[4.6.1]{BerLof76}, we have that, for $1<p<2$,  $\mathcal S^k_p(H):=[\p^k_N(H),\mathcal S^k_2(H)]_{2\theta}$, with $p(1-\theta)=1$. Thus, for $1<p<2$, we can obtain $c_{k,l}=1$. Similarly, for $2<p<\infty$ we have $$c_{k,l} \le\Big(\frac{(k+l)^{k+l}}{(k+l)!}\frac{k!}{k^k}\frac{l!}{l^l}\Big)^{1-\frac2{p}}. $$
\end{example}

\bigskip

\section{Holomorphic functions of $\u$-bounded type}\label{section hbu}

\bigskip
There is a natural way to associate to a holomorphy type $\u$ a class of holomorphic functions  on a Riemann domain $(X,q)$ spread over a  Banach space $E$. This space, denoted by $\hu(X)$, consists on all holomorphic functions that have positive $\u$-radius of convergence at each point of $X$, see for example \cite[Definition 2]{Din71(holomorphy-types)}. To give the precise definition, let us recall that if $f$ is a holomorphic function on $X$, then its $k$-th differential is defined by
$$\frac{d^kf(x)}{k!}:=\frac{d^k[f\circ(q|_{B_s(x)})^{-1}]}{k!}\big(q(x)\big).$$
\begin{definition}\rm
Let $\u=\{\u_k\}_k$ be a holomorphy type; $E$ a Banach space, and $(X,q)$ a Riemann domain  spread over a  Banach space $E$. A holomorphic function $f$ is of type $\u$ on $X$ if for each $x\in X,\, d^kf(x)$ belongs to  $\u_k(E)$ and
\begin{equation*}\label{hu}
 \lim_{k\rightarrow \infty}
 \Big\|\frac{d^kf(x)}{k!}\Big\|_{\u_k(E)}^{1/k}<\infty.
\end{equation*}
We denote by $\hu(X)$ the space of type $\u$ funtions on $X$.
 \end{definition}

We may also define a space of entire functions of bounded $\u$-type \cite{CarDimMur07,FavJat09} as the set of entire functions with infinite $\u$-radius of convergence at zero (and hence at every point). Similarly we can define the holomorphic functions of $\u$-bounded type on a ball of radius $r$ as
the holomorphic functions which have $\u$-radius of convergence equal $r$.
\begin{definition}\rm
Let $\u=\{\u_k\}_k$ be a holomorphy type; $E$ a Banach space, $x\in E$, and $r>0$.  We
define the space of \textbf{holomorphic functions of $\u$-bounded type} on $B_r(x)$ by
$$
\hbu(B_r(x))=\left\{f\in H(B_r(x))\ : \frac{d^kf(x)}{k!}
\in \u_k(E) \textrm{ and } \limsup_{k\rightarrow \infty}
\Big\|\frac{d^kf(x)}{k!}\Big\|_{\u_k}^{1/k}\le \frac1{r} \right\}.
$$
\end{definition}

We consider in $\hbu(B_r(x))$ the seminorms $p_s$, for
$0<s<r$, given by
$$
p_s(f)=\sum_{k=0}^{\infty}
\Big\|\frac{d^kf(x)}{k!}\Big\|_{\u_k}s^k,
$$
for all $f\in \hbu(B_r(x))$. Then it easy to show that $\Big(\hbu(B_r(x),F),\{p_s\}_{0<s<r}\Big)$ is a Fr\'echet space.

The following examples of spaces of holomorphic functions of bounded type on the unit ball $B_E$ were
already defined in the literature and can be seen as particular cases of the
above definition.

\begin{example} \rm
\begin{enumerate}
\item[(a)] If $\u$ is the sequence of ideals of continuous homogeneous polynomials,
then $\hbu(B_E)=H_b(B_E)$.
\item[(b)] If $\u$ is the sequence of ideals of weakly continuous on bounded sets
polynomials, then it is not difficult to see that  $\hbu(B_E)$ is the space $H_{wu}(B_E)$ of weakly uniformly continuous holomorphic functions on $B_E$-bounded sets. %weakly uniformly continuous holomorphic functions of
%bounded type $H_{bw}(E)$ defined by Aron in \cite{Aro79}.
\item[(c)] If $\u$ is the sequence of ideals of nuclear polynomials then $\hbu(B_E)$
is the space of holomorphic functions of nuclear bounded type
$H_{Nb}(B_E)$ defined by Gupta and Nachbin (see also \cite{Mat78}).
\item[(d)] If $\u$ is the sequence of ideals of extendible polynomials,
then by~\cite[Proposition
14]{Car01},  $\hbu(B_E)$ is the space of all $f\in H(B_E)$ such that,
for any Banach space $G\supset E $, there is an extension $ \tilde
f\in H_b(B_G)$ of~$f$.
\item[(e)] If $\u$ is the sequence of ideals of integral polynomials, then $\hbu(B_E)$ is the space of
integral holomorphic functions of bounded type $H_{bI}(B_E)$ defined
in \cite{DimGalMaeZal04}.
\end{enumerate}
\end{example}

\begin{remark}\rm In general, we have that $\hbu\subsetneq\hu\cap H_b$. Indeed, 
Dineen found in \cite[Example 9]{Din71(holomorphy-types)}  an entire function of bounded type on a Hilbert space $E$, $f\in H_b(E)$, such that  $f$ is of nuclear type on $E$, $f\in H_{N}(E)$, but $f$ is not an entire function of nuclear bounded type because $\lim_{n\to\infty}\|\frac{d^nf(0)}{n!}\|^{\frac1{n}}_N=1$. 
\end{remark}

We now define holomorphic functions of $\u$-bounded type on a Riemann domain $(X,q)$ spread over a Banach space.
If $f$ is of type $\u$ on $X$, and it has $\u$-radius of convergence greater than $s>0$ at $x\in X$, then we define
$$
p_s^x(f)=\sum_{k=0}^{\infty}
s^k\Big\|\frac{d^kf(x)}{k!}\Big\|_{\u_k(E)}.
$$
The holomorphic functions of $\u$-bounded type on $X$ are the holomorphic functions
on $X$ which are of the class $\hbu$ on every ball contained in $X$ and,
on each open $X$-bounded  set $A$, the seminorms
$p_s^x$ (with $x\in A$ and $B_s(x)\subset A$) are uniformly bounded.
\begin{definition}\label{def hbu(X)}\rm
A holomorphic function $f$ is of \textbf{$\u$-bounded type} on $(X,q)$ if:

(i) $f\circ(q|_{B_s(x)})^{-1}\in\hbu(q(B_s(x)))$ for every $s\le d_X(x)$.

(ii) For each open $X$-bounded  set $A$,
\begin{equation*}
 p_A(f):=\sup\{p_s^x(f):\,B_s(x)\subset A\}<\infty.
\end{equation*}
We denote by $\hbu(X)$ the space of all holomorphic functions of $\u$-bounded type on $(X,q)$.
\end{definition}
When $\u$ is the sequence $\p$ of ideals of continuous homogeneous polynomials, by the Cauchy inequalities, $\hbu(X)=H_b(X)$.
If $\u$ is the sequence of ideals of weakly continuous on bounded sets polynomials and $U$ is a balanced open set then $\hbu(U)=H_{wu}(U)$ (see, for example \cite[Proposition 1.1]{BurMor00}).
\begin{remark}\rm
Condition (i) in above definition states that $f$ is a holomorphic function of $\u$-bounded type on each ball contained in $X$. The space of holomorphic functions on $X$ that satisfy this condition is denoted by $\hdu(X)$
in \cite[Section 3.2.6]{Mur10}. This resembles much the definition given in
\cite[Section 3]{DinVen04} of the space $H_d(X)$ (indeed $\hdu(X)=H_d(X)$ when $\u=\p$).
The seminorms $\{p_s^x:\,0<s<d_X(x),\, x\in X\}$ define a topology on $\hdu(X)$ which is always complete but not necessarily a Fr\'echet space topology unless $E$
is separable. In that case we may follow the proof of \cite[Proposition
3.2]{DinVen04} to show that $H_{d\u}(X)$ is a Fr\'echet space. Most of the results in this article remain true if we replace $\hbu$ by $\hdu$.
\end{remark}

% VECTORIAL
% \begin{definition}
% Let $F$ be a Banach space. We will say that a mapping is in $\hdu(X,F)$ if it is
% $\u$-holomorphic of bounded type on each ball in $X$, that is,
% \begin{equation}\label{hdu}
% % \hdu(X,F):=\left\{f\in H(X,F)\ : \textrm{for every }x\in X,\, f\circ
% % (p|_{B^x})^{-1}\in\hbu(p(B^x),F)\right\}.
% \hdu(X,F):=\left\{f\in H(X,F)\ : \textrm{for every }x\in X,\, f\circ
% (p|_{B_s(x)})^{-1}\in\hbu(p(B_s(x)),F),\,\forall s<d_X(x)\right\}.
% \end{equation}
% \end{definition}
% We define the seminorms $p_s^x(f)$ by
% $$
% p_s^x(f)=\sum_{k=0}^{\infty}
% s^k\Big\|\frac{d^kf(x)}{k!}\Big\|_{\u_k(E,F)},
% $$
% for $0<s<d_X(x)$, $x\in X$ and where
% $\frac{d^kf(x)}{k!}:=\frac{d^k[f\circ(p|_{B_s(x)})^{-1}]}{k!}\big(p(x)\big)$.
% These seminorms define a topology on $\hdu(X)$ which is always complete (see
% Remark~\ref{hdu complete}) but not necessarily a Fr\'echet topology unless $E$
% is separable. In that case we may copy the proof of \cite[Proposition
% 3.2]{DinVen04} to obtain:
% \begin{proposition}
% Let $\u$ be a coherent sequence and $(X,p)$ be a connected Riemann domain over a
% separable Banach space $E$, then $H_{d\u}(X,F)$ is a Fr\'echet space.
% \end{proposition}

\begin{proposition}\label{hbudeX frechet}
The seminorms $\{p_A:\, A\textrm{ open and }X\textrm{-bounded}\}$ define a
Fr\'echet space topology on $\hbu(X)$.
\end{proposition}
\begin{proof}
It is clear that the topology may be described with the countable set of
seminorms $\{p_{X_n}\}_{n\in\mathbb N}$, where $X_n=\{x\in X:\|q(x)\|<
n\textrm{ and }d_X(x)>\frac1{n}\}$, so we only need to prove completeness. Let
$(f_k)_k$ be a Cauchy sequence in $\hbu(X)$, then it is a Cauchy sequence in
$H_b(X)$, so there exists a function $f\in H_b(X)$ which is limit (uniformly
in $X$-bounded sets) of the $f_k$'s. 

Let $x\in X$ and
$r\le d_X(x)$. Then $(f_k\circ (q|_{B_r(x)})^{-1})_k$ is a Cauchy sequence in
$\hbu(B_r(q(x)))$ which converges pointwise to $f\circ (q|_{B_r(x)})^{-1}$. Since
$\hbu(B_r(q(x)))$ is complete we have that $f\circ (q|_{B_r(x)})^{-1}$
belongs to $\hbu(B_r(q(x)))$, and thus $f$ satisfies $(i)$ of Definition \ref{def hbu(X)}. Moreover, for each $k$, $p_s^x(f-f_k)\le\limsup_jp_s^x(f_j-f_k)$ for every $s<d_X(x)$. Thus, if $A$ is an $X$-bounded set, 
$$
p_A(f-f_k)=\sup_{B_s(x)\subset A}p_s^x(f_j-f_k)\le \limsup_j\sup_{B_s(x)\subset A}p_s^x(f_j-f_k)=\limsup_jp_A(f_j-f_k),
$$
which goes to $0$ as $k\to\infty$. Therefore, $f$ is in $\hbu(X)$ and $(f_k)$ converges to $f$ in $\hbu(X)$.
\end{proof}

\section{Multiplicative sequences}\label{section multiplicative}

In this section we show that under a condition on $\u$ which is satisfied for most of the commonly used polynomial ideals, the space $\hbu(X)$ is a locally $m$-convex Fr\'echet algebra.

\begin{definition}\rm
Let $\s$ be a sequence of  polynomial ideals. We say that $\s$ is {\bf multiplicative} at $E$ if  there exist
constants $c_{k,l}>0$ such that for
each $P\in\u_k(E)$ and $Q\in\u_l(E)$, we have that
$PQ\in\u_{k+l}(E)$ and
$$
\|PQ\|_{\u_{k+l}(E)}\le c_{k,l}\|P\|_{\u_k(E)}\|Q\|_{\u_l(E)}.
$$
\end{definition}
If $\{\u_k\}$ is a multiplicative sequence then the sequence of adjoint ideals $\{\u^*_k\}$ is a holomorphy type with the same constants, and it is moreover a weakly differentiable sequence (see Remark \ref{w-dif dual a mult}).

In \cite{CarDimMur} we studied multiplicative sequences with constants $c_{k,l}\le M^{k+l}$ for some constant $M\ge1$ and proved that in this case the space of entire functions of $\u$-bounded type is an algebra. 
 To obtain algebras of holomorphic functions on balls or on Riemann domains we need to have more restrictive bounds for $c_{k,l}$.
Actually, we impose the constants $c_{k,l}$ to satisfy the inequality (\ref{constantes}) for every $k,l\in\mathbb N$.
\begin{remark}\rm \label{stirling} %J. Sandor and L. Debnath, On certain
%inequalities involving the constant e and their applications, J. Math. Anal.
%Appl. 249, pp. 569-582,(2000)
Stirling's Formula states that $e^{-1}n^{n+1/2}\le e^{n-1}n!\le
n^{n+1/2}$ for every $n\ge1$, so we have that
\begin{equation}\label{eq stirling}
\frac{(k+l)^{k+l}}{(k+l)!}\frac{k!}{k^k}\frac{l!}{l^l}\le e^2
\Big(\frac{kl}{k+l}\Big)^{1/2}.
\end{equation}
As a consequence, if $\u$ is multiplicative with $c_{k,l}$ as in (\ref{constantes}) then, for each $\varepsilon>0$ there exists a
constant $c_\varepsilon>0$ such that for every $k,l\in\mathbb N$, $P\in\u_k(E)$ and $Q\in\u_l(E)$, we have,
$$\|PQ\|_{\u_{k+l}(E)}\le c_\varepsilon
(1+\varepsilon)^{k+l}\|P\|_{\u_k(E)}\|Q\|_{\u_l(E)}.$$
\end{remark}
We will show bellow that every example of holomorphy type mentioned in Section \ref{section prelim}
is a
multiplicative sequence with constants that satisfy (\ref{constantes}).
Let us see before that in this case $\hbu(X)$ is a locally $m$-convex Fr\'echet algebra, that is, the topology may be given by a sequence of submultiplicative
seminorms. By a theorem by Mitiagin, Rolewicz and Zelazko \cite{MitRolZel62}, it suffices to show that they are (commutative) $B_0$-algebras\footnote{Recall that a $B_0$-algebra is a complete metrizable topological algebra such that the topology is given by means of an increasing sequence $\|\cdot\|_1\le\|\cdot\|_2\le\dots$ of seminorms satisfying that $\|xy\|_j\le C_i\|x\|_{j+1}\|y\|_{j+1}$ for every $x,y$ in the algebra and every $j\ge1$, where $C_i$ are positive constants. It is possible to make  $C_i=1$ for all $i$ \cite{Zel94}.} %(Encyclopaedia of mathematics, M. Hazewinkel, p.257)}
and that functions in $H(\mathbb C)$ operate on $\hbu$ (that is, if $g(z)=\sum_k a_kz^k$ belongs to $H(\mathbb C)$ and $f\in\hbu$, then $\sum_k a_kf^k$ belongs to $\hbu$). We consider first the case of a ball and the whole space.
\begin{proposition}\label{hbu de la bola algebra}
Suppose that $\u$ is multiplicative with $c_{k,l}$ as in (\ref{constantes}), and $E$ a Banach space. Then,
\begin{itemize}
  \item[$(i)$]  for each $x\in E$ and $r>0$, $H_{b\u}(B_r(x))$ is a locally
$m$-convex Fr\'echet algebra.
  \item[$(ii)$] $H_{b\u}(E)$ is a locally $m$-convex Fr\'echet algebra.
\end{itemize}
\end{proposition}
\begin{proof}
We just prove $(i)$ because $(ii)$ follows similarly. We will show this for $r=1$ and $x=0$, that is for $B_r(x)=B_E$. The general
case follows by translation and dilation. We already know that $H_{b\u}(B_E)$ is
a Fr\'echet space. Let us first show that it is a $B_0$-algebra.

Let $f=\sum_k P_k$ and $g=\sum_k Q_k$ be functions in $H_{b\u}(B_E)$. We must
show that $\frac{d^nfg(0)}{n!}$ belongs to $\u_n(E)$ and that
$p_s(fg)=\sum_{n=0}^\infty s^n\big\|\frac{d^nfg(0)}{n!}\big\|_{\u_n(E)}<\infty$
for every $s<1$. Since $\frac{d^nfg(0)}{n!}=\sum_{k=0}^nP_kQ_{n-k}$ and $\u$ is
multiplicative, $\frac{d^nfg(0)}{n!}$ belongs to $\u_n(E)$. On the other
hand, by (\ref{eq stirling}),
\begin{eqnarray*}
\sum_{n=0}^\infty s^n\big\|\frac{d^nfg(0)}{n!}\big\|_{\u_n(E)} & \le &
e^2\sum_{n=0}^\infty s^n\sum_{k=0}^n
\Big(\frac{k(n-k)}{n}\Big)^{1/2}\|P_k\|_{\u_k(E)}\|Q_{n-k}\|_{\u_{n-k}(E)} \\
& = & e^2 \sum_{k=0}^\infty \sqrt{k}s^k\|P_k\|_{\u_k(E)}\sum_{n=k}^\infty
s^{n-k} \Big(\frac{n-k}{n}\Big)^{1/2}\|Q_{n-k}\|_{\u_{n-k}(E)} \\
& \le & e^2p_s(g)\sum_{k=0}^\infty \sqrt{k}s^k\|P_k\|_{\u_k(E)}.
\end{eqnarray*}
Therefore, for each $\varepsilon>0$ there exists a constant
$c=c(\varepsilon,s)>1$ such that
\begin{equation}\label{seminormas hyper}
 p_s(fg)=\sum s^n\big\|\frac{d^nfg(0)}{n!}\big\|_{\u_n(E)}  \le c
p_s(g)p_{s+\varepsilon}(f).
\end{equation}
Define, for each $n\ge1$, $s_n=1-\frac1{2^n}$ and
$c_n=c(\frac1{2^{n+1}},1-\frac1{2^n})$. Then, for every $f,g\in\hbu(B_E)$, we have
\begin{equation}\label{hbudeu b0 alg}
p_{s_n}(fg)\le c_n p_{s_n}(g)p_{s_{n+1}}(f)\le c_np_{s_{n+1}}(f)p_{s_{n+1}}(g).
\end{equation}
Since the seminorms $p_{s_n}$ determine the topology of $\hbu(B_E)$, we conclude that $\hbu(B_E)$ is a $B_0$-algebra. %en \cite{MitRolZel62} definen B_0 algebras con constantes =1, pero ver Zelazko, Concerning entire functions in Bo-algebras, Studia (110, 1994)
 Note also that 
(\ref{hbudeu b0 alg}) implies that $p_{s_n}(f^k)\le c_n^{k-1}
p_{s_{n+1}}(f)^{k-1} p_{s_n}(f)\le c_n^{k} p_{s_{n+1}}(f)^{k}.$ Take now an
entire function $h\in H(\mathbb C)$, $h(z)=\sum a_kz^kz$. Then for
$f\in\hbu(B_E)$,
$$
p_{s_n}\big(\sum_{k=N}^Ma_kf^k\big) \le \sum_{k=N}^Ma_kp_{s_n}(f^k) \le
\sum_{k=N}^Ma_k\big(c_n p_{s_{n+1}}(f)\big)^{k},
$$
which tends to 0 as $N,M$ increase because $h$ is an entire function. This means
that entire functions operate in $\hbu(B_E)$.
Therefore \cite[Theorem 1]{MitRolZel62} implies that $\hbu(B_E)$ is locally
$m$-convex.
\end{proof}

\begin{remark}\rm
 Michael's conjecture \cite{Mic52} states that on any Fr\'echet algebra, every character is continuous. Adapting some of the ideas in \cite{Cra71,Cla75,Sch81}, Mujica showed (see \cite[Section 33]{Muj86} or \cite{Muj06}) that if every character on $H_b(E)$ is continuous for some infinite dimensional Banach space $E$ then the conjecture is true for every commutative Fr\'echet algebra. As a corollary of his results we may deduce that the same is true for the Fr\'echet algebra $\hbu(E)$ for any multiplicative sequence $\u$ with constants as in (\ref{constantes}). 
% Indeed, if $\mathcal A$ is a commutative, complete, Hausdorff locally $m$-convex algebra and $\psi$ is a discontinuous character on $\mathcal A$, then it is easy to obtain a sequence $(x_n)$ in $\mathcal A$ such that $|\psi(x_n)|>2^n$ and that $\sum_n\sqrt{p(x_n)}<\infty$ for each continuous seminorm $p$ on $\mathcal A$. Mujica showed that if $(e_n)$ is a basic sequence in $E$ and $(e_n')$ are the coordinate functionals (extended to $E$), then there exists a continuous homomorphism $T:H_b(E)\to\mathcal A$ such that $T(e_n')=x_n$. Thus $\psi\circ T$ is a discontinuous character on $H_b(E)$. Since the inclusion $j:\hbu(E)\to H_b(E)$ is continuous and since $(e_n')$ is a bounded sequence in $\hbu(E)$, we conclude that $\psi\circ T\circ j$ is a discontinuous character on $\hbu(E)$.
We would like however to sketch an alternative proof of this fact as consequence of a result by Ryan \cite{Rya87} on the convergence of monomial expansions for entire functions on $\ell_1$. Theorem 3.3 in \cite{Rya87} states that for each $f\in H_b(\ell_1)$ there exist unique complex coefficients $(a_m)_{m\in\mathbb N^{(\mathbb N)}}$ such that $f(z)=\sum_{m\in\mathbb N^{(\mathbb N)}}a_mz^m$ , where a multi-index $m\in\mathbb N^{(\mathbb N)}$ is a sequence of are non-negative integers such that only a finite number of them are non-zero and where $z^m=\Pi_{j\in\mathbb N}z_j^{m_j}$. The convergence of the monomial expansion of $f$ is absolute for every $z\in\ell_1$ and uniform on bounded sets of $\ell_1$. Moreover, the coefficients satisfy 
\begin{equation}\label{coef monomial}
\lim_{|m|\to\infty}(|a_m|m^m/|m|^{|m|})^{1/|m|}=0.
\end{equation}
Conversely, any such coefficients define a function in $H_b(\ell_1)$. 
 Let $\mathcal A$ be a commutative, complete, Hausdorff locally $m$-convex algebra which has an unbounded character $\psi$. We may suppose that $\mathcal A$ has unit $e$. Let $\mathbf x=(x_n)$ be a sequence in $\mathcal A$ such that $\sum_np(x_n)<\infty$ for each continuous seminorm $p$ on $\mathcal A$, and that $(\psi(x_n))$ is unbounded. For $m=(m_1,m_2,\dots)\in\mathbb N^{(\mathbb N)}$, let $\mathbf x^m=\Pi_{j\in\mathbb N}x_j^{m_j}$, where $x^0=e$ for every $x\in\mathcal A$. Let $T:H_b(\ell_1)\to\mathcal A$ be defined by $Tf=\sum_{m\in\mathbb N^{(\mathbb N)}}a_m\mathbf x^m$, where $(a_m)_{m\in\mathbb N^{(\mathbb N)}}$ are the coefficients of $f$. Note that $g(z)=\sum_{m\in\mathbb N^{(\mathbb N)}}|a_m|z^m$ defines a function in $H_b(\ell_1)$ because its coefficients 
 satisfy (\ref{coef monomial}). Thus, for a continuous seminorm $p$ on $\mathcal A$, $$\sum_{m\in\mathbb N^{(\mathbb N)}}|a_m|p(\mathbf x^m)\le \sum_{m\in\mathbb N^{(\mathbb N)}}|a_m|\Pi_{j\in\mathbb N}p(x_j)^{m_j}=g((p(x_j))_j),$$ which implies that $T$ is well defined. Clearly $T$ is an algebra homomorphism. 

% Moreover, Ryan defined the seminorms $p_R(f)=\sup\{R^{|m|}|a_m|m^m/|m|^{|m|}:\,m\in\mathbb N^{(\mathbb N)}\}$ for $R>0$ and proved in \cite[Proposition 3.5]{Rya87} that the topology generated by these seminorms is equivalent to the usual topology in $H_b(\ell_1)$. Let $R\ge\|(p(x_j))_j\|_1$, then $p(Tf)\le g((p(x_j))_j)\le \sup\{g(y):\,\|y\|_1\le R\}\le cp_{R'}(g)=cp_{R'}(f)$, for some constant $c>0$ and $R'>0$. Therefore $T$ is continuous. 

Note also that $T(e_j')=x_j$, where $e_j'$ is the $j$-th coordinate functional in $\ell_1$. Therefore $\psi\circ T$ is a discontinuous character on $H_b(\ell_1)$.

Let now $E$ be any Banach space and let $(y_j)$, $(y_j')$ be a biorthogonal sequence in $E$ with $\|y_j\|<1$ and $(y_j')$ bounded. Let $M$ be the closed space spanned by the $y_j$'s and let $(z_k)$ be a dense sequence in the unit ball of $M$ that contains the sequence $(y_j)$. Say  $y_j=z_{n_j}$. Then the linear map which sends each $e_k\in\ell_1$ to $z_k$ induces an isomorphism from a quotient of $\ell_1$ to $M$.  Consider the following mapping $R=R_4R_3R_2R_1:H_b(E)\to  H_b(\ell_1)$ 
$$ \hbu(E)\overset{R_1}{\longrightarrow}H_b(E)\overset{R_2}{\longrightarrow} H_b(M) \overset{R_3}{\longrightarrow}  H_b(\ell_1)\overset{R_4}{\longrightarrow} H_b(\ell_1),$$
where $R_1$ is the inclusion, $R_2$ is the restriction from $E$ to $M$, $R_3$ is the composition with the quotient map and $R_4$ is the restriction to the closed space spanned by the $e_{n_j}$'s (we identify this space with $\ell_1$). Then $R(y_j')$ is a linear functional on $\ell_1$. Moreover, if $\overline{z}$ denotes the class of $z\in\ell_1$ in the quotient of $\ell_1$ isomorphic to $M$, then $\overline{e_{n_k}}=z_{n_k}=y_k$. Thus $R(y_j')(e_k)=R_3R_2R_1(y_j')(e_{n_k})=R_2R_1(y_j')(\overline{e_{n_k}})=y_j'(y_k)=\delta_{kj}$, that is, $R(y_j')=e_j'$. 

Since $(y_j')$ is a bounded sequence in $\hbu(E)$ and $R$ is an algebra homomorphism, $\psi\circ T\circ R$ is a discontinuous character on $\hbu(E)$.
\end{remark}

We will now prove that $\hbu(X)$ is a locally $m$-convex Fr\'echet algebra, for $(X,q)$ an arbitrary Riemann domain over $E$. We first need the following.
\begin{lemma}\label{lema B_s+delta}
 Let $A$ be an $X$-bounded set with $d_X(A)\ge\delta$. If $B_s(x)\subset A$ then $B_{s+\delta}(x)$ exists.
\end{lemma}
\begin{proof}
Suppose that we can show that $B_{s+\frac{\delta}{4}}(x)$ exists. Then $B_{s+\frac{\delta}{4}}(x)$ is contained in the $X$-bounded set $A_\frac{\delta}{4}=\bigcup_{x\in A}B_{\frac{\delta}{4}}(x)$, and $d_X(A_\frac{\delta}{4})\ge\frac{3\delta}{4}$. Applying the result proved to $B_{s+\frac{\delta}{4}}(x)$ and $A_\frac{\delta}{4}$ we have that $B_{s+\frac{\delta}{4}(1+\frac{3}{4})}(x)$ exists. Applying the same process $n+1$ times, we have that $B_{s+\frac{\delta}{4}(1+\frac{3}{4}+\dots+(\frac{3}{4})^{n})}(x)$ exists. Clearly $\bigcup_{n\in\mathbb N} B_{s+\frac{\delta}{4}(1+\frac{3}{4}+\dots+(\frac{3}{4})^{n})}(x)$ is $B_{s+\delta}(x)$.
Thus, it suffices to prove that $B_{s+\frac{\delta}{4}}(x)$ exists.

 Let $C=\cup_{y\in B_s(x)}B_\frac{\delta}{4}(y)$. Then $q(C)=\cup_{y\in B_s(x)}B_\frac{\delta}{4}(q(y))=B_{s+\frac{\delta}{4}}(q(x))$. If we show that $q|_C$ is injective then $C=B_{s+\frac{\delta}{4}}(x)$.

Take $x_0\ne x_1$ in $C$. Then there exist $y_0,y_1\in B_s(x)$ such that $x_j\in B_\frac{\delta}{4}(y_j)$, $j=0,1$. If $B_\frac{\delta}{4}(y_0)\cap B_\frac{\delta}{4}(y_1)\ne\emptyset$, then $x_0$ and $x_1$ are in $B_\delta(y_0)$. Since $q$ is injective on $B_\delta(y_0)$, we have that $q(x_0)\ne q(x_1)$. On the other hand, if $B_\frac{\delta}{4}(y_0)\cap B_\frac{\delta}{4}(y_1)=\emptyset$, then $B_{s}(q(x))\cap B_\frac{\delta}{4}(q(y_0))\cap B_\frac{\delta}{4}(q(y_1))=\emptyset$, because $q$ is injective on $B_s(x)$.
But, since $q(y_0)$ and $q(y_1)$ are in $B_s(q(x))$, we can conclude that $B_\frac{\delta}{4}(q(y_0))\cap B_\frac{\delta}{4}(q(y_1))=\emptyset$ and thus $q(x_0)\ne q(x_1)$.
\end{proof}

\begin{theorem}
Suppose that $\u$ is a multiplicative sequence with constants as in (\ref{constantes}) and let $(X,q)$ be a Riemann domain over $E$.
Then $\hbu(X)$ is a locally $m$-convex Fr\'echet algebra.
\end{theorem}
\begin{proof}

We know from Proposition \ref{hbudeX frechet} that $\hbu(X)$ is a Fr\'echet

Let $X_n$ denote the set $\{x\in X:\, \|q(x)\|< n \textrm{ and }d_X(x)>\frac1{n}\}$. If $B_s(x)$ is contained in $X_n$ and $\varepsilon_n<\frac1{n}-\frac1{n+1}$ then $B_{s+\varepsilon_n}(x)\subset X_{n+1}$ (note that $B_{s+\varepsilon_n}(x)$ exists by Lemma \ref{lema B_s+delta}). Proceeding as in Proposition \ref{hbu de la bola algebra}, we can show that for every $f,g\in\hbu(X)$, $p_s^x(fg)\le c_n p_{s+\varepsilon_n}^x(f)p_{s}^x(g)$, which implies that $p_{X_n}(fg)\le c_n p_{X_{n+1}}(f)p_{X_n}(g)$. Thus, $\hbu(X)$ is a commutative $B_0$-algebra. Moreover, we also have that $p_{X_n}(f^k)\le c_n^k p_{X_{n+1}}(f)^k$, which implies that entire functions operate on $\hbu(X)$. Therefore by \cite[Theorem 1]{MitRolZel62} we conclude that $\hbu(X)$ is a locally $m$-convex algebra.
\end{proof}

To finish this section we present some examples of multiplicative sequences with constants $c_{k,l}$ as in (\ref{constantes}).
\begin{example} It is clear that the following sequences are
multiplicative with constants $c_{k,l}=1$.

$i)$ $\p$, of continuous homogeneous polynomials,

$ii)$ $\p_{w}$, of weakly continuous on bounded sets polynomials,

$iii)$ $\p_A$, of approximable polynomials,

$iv)$ $\p_e$, of extendible polynomials,
\end{example}

\begin{example} The sequence $\p_I$ of integral polynomials.\rm\newline
\smallskip
It was shown in \cite[Example 2.3 (c)]{CarDimMur} that if $P,Q$ are homogenous integral polynomials then $PQ$ is integral with $\|PQ\|_I\le
\frac{(k+l)^{k+l}}{(k+l)!}\frac{k!}{k^k}\frac{l!}{l^l}\|P\|_I\|Q\|_I$
\end{example}

\begin{example}\label{ejemplo nuclear hypermultiplicativas} The sequence
$\p_N$ of nuclear polynomials.\rm\newline\smallskip
Proposition 2.6 in \cite{CarDimMur} implies that if $\{\u_k\}$ is a multiplicative sequence then the sequences of maximal and minimal hulls, $\{\u_k^{max}\}$ and $\{\u_k^{min}\}$, are multiplicative with the same constants. Since nuclear polynomials are the minimal ideal associated to integral
polynomials (see for example \cite[3.4]{Flo01}), we have that they form a multiplicative sequence with constants as in (\ref{constantes}). See also \cite[Exercise 2.63]{Din99}.

  Note that, as a
consequence of Proposition \ref{hbu de la bola algebra}, the space of nuclearly
entire functions of bounded type is a locally $m$-convex Fr\'echet algebra.
\end{example}

The sequences $\p$, $\p_N$,$\p_e$, $\p_I$ and $\p_A$ are particular cases of the following.
\begin{example}
Let $\{\alpha_k\}_k$ be any of the sequences of natural symmetric tensor norms.
Then the sequences  $\{\u^{max}_k\}_k$ and $\{\u^{min}_k\}_k$ of maximal and
minimal ideals associated to $\{\alpha_k\}_k$ are multiplicative with constants $c_{k,l}$ as in (\ref{constantes}).\rm\newline
This follows from the inequalities
$$
\pi_{k+l}(\sigma(s\otimes t))\le
\frac{(k+l)^{k+l}}{(k+l)!}\frac{k!}{k^k}\frac{l!}{l^l}
\pi_k(s)\pi_l(t),\quad\textrm{  }\quad\varepsilon_{k+l}(\sigma(s\otimes t))\le
\varepsilon_k(s)\varepsilon_l(t)
$$
for every $s\in\TS{k}_{}E'$, $t\in\TS{l}_{}E'$ together with Proposition 2.6 and Lemma 2.9 of \cite{CarDimMur}.
\end{example}

 \begin{example}
The sequence $\mathcal M_r$ of multiple $r$-summing polynomials is multiplicative with constants $c_{k,l}=1$.
\end{example}
\begin{proof}
 Let $P\in \mathcal M_r^k(E), Q\in \mathcal M_r^l(E)$, then
$$ (PQ)^{\vee}
(x_1,\dots,x_{k+l})=\frac{k!}{(k+l)!}\sum_{\overset{s_1,\dots,s_l=1}{^{s_1\ne\dots\ne s_l}}}^{k+l} \v
P(x_1,\overset{^{s_1\dots s_l}}{\dots},x_{k+l}) \v
Q(x_{s_1},\dots,x_{s_{l}})$$ where
$\v
P(x_1,\overset{^{s_1\dots s_l}}{\dots},x_{k+l})$ means that coordinates $x_{s_1},\dots,x_{s_{l}}$ are omitted.

Take $(x^{i_{j}}_{j})_{j=1}^{m_j}\subset E$, for $j=1,\dots,k+l$, such that $w_{r}((x^{i_{j}}_{j}))=1$.
Then, using the triangle inequality for the $\ell_r$-norm,
\begin{eqnarray*}
& & \left(  \sum_{i_{1},\dots,i_{k+l}=1}^{m_{1},\dots,m_{k+l}}  \left|
(PQ)^{\vee} (x^{i_{1}}_{1},\dots,x^{i_{k+l}}_{k+l})\right|^{r}
\right)^{\frac{1}{r}} \le \\
&  & \ \quad\le  \frac{k!}{(k+l)!}\sum_{\overset{s_1,\dots,s_l=1}{^{s_1\ne\dots\ne s_l}}}^{k+l}
\left(\sum_{i_{1},\dots,i_{k+l}=1}^{m_{1},\dots,m_{k+l}}
\left|\v P(x_1,\overset{^{s_1\dots s_l}}{\dots},x_{k+l})\right|^r \left|\v
Q(x_{s_1},\dots,x_{s_{l}})\right|^r
\right)^{1/r} \\
&  & \ \quad\le \frac{k!}{(k+l)!}\sum_{\overset{s_1,\dots,s_l=1}{^{s_1\ne\dots\ne s_l}}}^{k+l} \left(
\sum_{i_{s_1},\dots,i_{s_l}=1}^{m_{s_1},\dots,m_{s_l}}\left|\v
Q(x_{s_1},\dots,x_{s_{l}})\right|^r  \|P\|_{\mathcal M_r ^k}^r\right)^{1/r}\\
&  & \ \quad\le \frac{k!}{(k+l)!}\sum_{\overset{s_1,\dots,s_l=1}{^{s_1\ne\dots\ne s_l}}}^{k+l} \|P\|_{\mathcal M_r^k}\|Q\|_{\mathcal M_r^l} =\|P\|_{\mathcal M_r^k}\|Q\|_{\mathcal M_r^l}.\end{eqnarray*}
Hence, $PQ$ is multiple $r$-summing with $\|PQ\|_{\mathcal M_r^{k+l}}\le \|P\|_{\mathcal M_r^k}\|Q\|_{\mathcal M_r^l}$.
\end{proof}

\begin{example}\label{HS multip}
The sequence $\mathcal S_2$ of Hilbert-Schmidt polynomials.
\newline
\smallskip\rm
Petersson proved in \cite{Pet01} that the ideals of Hilbert-Schmidt polynomials form a multiplicative sequence with $c_{k,l}\le \sqrt{2^{k+l}}$. The duality between multiplicativity and weak differentiability (see Remark \ref{w-dif dual a mult} and \cite[Proposition 3.16]{CarDimMur}) will allow us to prove in Example \ref{HS w-diff} that actually $c_{k,l}=1$.
% \textcolor{red}{En realidad Petersson usa otra norma: $\|P\|_{H}=\sqrt{n!}\|P\|_{S_2}$. Su cota traducida a nuestra norma (en realidad la cota que sale de la demostracion que es un $\binom{k+l}{k}$ en vez de $2^{k+l}$) dar\'ia $\|PQ\|_{S_2}\le\sqrt{\binom{k+l}{k}}\|P\|_{S_2}\|Q\|_{S_2}$. O sea que el $c_{k,l}$ al que \'el llega es mejor que el $2^{k+l}$ que estoy diciendo pero peor que 1. (lo que si estÃ¡ mal es lo que puse en la tesis! ahi dije que $c_{k,l}=\sqrt{e}^{k+l}$ mejoraba lo que \'el hizo, pero su $c_{k,l}$ es menor a $\sqrt{2}^{k+l}$).}
\end{example}

\begin{example}\label{schatten multip}
The sequence $\mathcal S_p$ of $p$-Schatten-von Neumann polynomials.\rm
\newline
\smallskip
Using the Reiteration theorem for the complex interpolation method and the previous examples we deduce that $\{\mathcal S_p^k\}$ is a multiplicative sequence with constants $c_{k,l}=1$ for $2<p<\infty$ and $$c_{k,l} \le\Big(\frac{(k+l)^{k+l}}{(k+l)!}\frac{k!}{k^k}\frac{l!}{l^l}\Big)^{\frac2{p}-1} $$ for $1<p<2$.
\end{example}

\section{Analytic structure on the spectrum}

Let $(X,q)$ be a Riemann domain over a Banach space $E$. In this section we  prove that, under fairly general conditions, the spectrum of the algebra $\hbu(X)$ may be endowed with a structure of Riemann domain spread over the bidual $E''$. This will extend some of the results in \cite{AroGalGarMae96,CarDimMur}.

As in the case of $H_b$ studied in \cite{AroGalGarMae96} or entire functions in $\hbu(E)$ studied in \cite{CarDimMur}, extensions to the bidual will be crucial, so we will need them to behave nicely. Indeed, we will need the following two conditions which were already defined in \cite{CarDimMur}.
\begin{definition}\label{AB closed}\rm
Let $\u$ be a sequence of ideals of polynomials.
We say that $\u$ is 
 $\mathbf{AB}$-\textbf{closed} if for each Banach space $E$, $k\in\mathbb N$ and
$P\in\u_k(E)$ we have that $AB(P)$ belongs to $\u_k(E'')$ and
$\|AB(P)\|_{\u_k(E'')}\le\|P\|_{\u_k(E)}$, where $AB$
denotes the Aron-Berner extension \cite{AroBer78}.
\end{definition}
Recall that an \textbf{Arens extension} of a $k$-linear form $A$ on $E$ is an extension to the bidual $E''$ obtained by $w^*$-continuity on each variable in some order.
\begin{definition}\rm
We say that a sequence of ideals of polynomials $\u$ is \textbf{regular} at $E$ if, for every
$k$ and every $P$ in $\u_k(E)$, we have that every Arens extension  of  $\v P$ is
symmetric.
 We say that the sequence $\u$ is regular if it is regular at $E$ for every Banach space $E$.
\end{definition}
All the examples given Section \ref{section multiplicative} are known to be $AB$-closed (see \cite{CarDimMur}).
 All the examples given Section \ref{section multiplicative} but $\p$ and $\mathcal M_r$ are known to be regular at any Banach space. Any sequence of polynomial ideals is regular at a symmetrically regular Banach space.

The main purpose of this section is to prove the following result.
\begin{theorem}\label{mbu dom riemann}
Let $\u$ be a multiplicative holomorphy type with constants as in (\ref{constantes}) which is regular at $E$ and $AB$-closed. Then
 $(M_{b\u}(X),\pi)$ is a Riemann domain over $E''$.
\end{theorem}
First we will need some preliminary lemma. 
 \begin{lemma}\label{lema0}
  Let $\u$ be a holomorphy type with constants as in (\ref{constantes}). Define $\delta_{(w_{1}^{k_1},\dots,w_{h}^{k_h})}\in\u_k(E)'$~by, $$\delta_{(w_{1}^{k_1},\dots,w_{h}^{k_h})}(Q)=\v Q(w_{1}^{k_1},\dots,w_{h}^{k_h}),$$ where, $w_{1},\dots,w_{h}\in E$ and $k_1,\dots,k_h\in\mathbb N$ are such that $k_1+\dots+k_h=k$.  Then for any $P\in\u_{k+l}(E)$, the polynomial  $R(x):=\delta_{(w_{1}^{k_1},\dots,w_{h}^{k_h})}(P_{x^l})$ is in $\u_l(E)$ and
$$
\|R\|_{\u_l(E)}\le\frac{(k+l)^{k+l}k_1!\dots k_h!l!}{(k+l)!k_1^{k_1}\dots k_h^{k_h}l^l}\|w_1\|^{k_1}\dots\|w_h\|^{k_h}\|P\|_{\u_{k+l}(E)}.
$$
If $\u$ is also $AB$-closed and regular at $E$ then the above statements hold for any $w_{1},\dots,w_{h}\in E''$, where $\delta_{(w_{1}^{k_1},\dots,w_{h}^{k_h})}\in\u_k(E)'$ is defined by $\delta_{(w_{1}^{k_1},\dots,w_{h}^{k_h})}(Q)=\v {AB(Q)}(w_{1}^{k_1},\dots,w_{h}^{k_h})$.
 \end{lemma}
\begin{proof}
We proceed by induction on $h$. For $h=1$, this is a consequence of $\u$ being a holomorphy type. Suppose that it holds for $h=n$ and let $Q(x)=\delta_{(w_{1}^{k_1},\dots,w_{n}^{k_n})}(P_{x^{k_{n+1}+l}})$. Then $Q$ belongs to $\u_{k_{n+1}+l}(E)$ and
\begin{equation}\label{eq1 remark lema1}
 \|Q\|_{\u_{k_{n+1}+l}(E)}\le\frac{(k+l)^{k+l}k_1!\dots k_n!(k_{n+1}+l)!}{(k+l)!k_1^{k_1}\dots k_n^{k_n}(k_{n+1}+l)^{k_{n+1}+l}}\|w_1\|^{k_1}\dots\|w_n\|^{k_n}\|P\|_{\u_{k+l}(E)}.
\end{equation}
Thus, $x\mapsto\delta_{w_{n+1}^{k_n+1}}(Q_{x^l})=\delta_{(w_{1}^{k_1},\dots,w_{n+1}^{k_{n+1}})}(P_{x^{l}})$ belongs to $\u_l(E)$ and
\begin{equation}\label{eq2 remark lema1}
 \|x\mapsto\delta_{w_{n+1}^{k_n+1}}(Q_{x^l})\|_{\u_l(E)}\le\frac{(k_{n+1}+l)^{k_{n+1}+l}k_{n+1}!l!}{(k_{n+1}+l)!k_{n+1}^{k_{n+1}}l^l}\|w_{k_{n+1}}\|^{n+1}\|Q\|_{\u_{k_{n+1}+l}(E)}.
\end{equation}
Putting (\ref{eq1 remark lema1}) and (\ref{eq2 remark lema1}) together, we obtain our claim. The last statement follows similarly.
 \end{proof}
\begin{lemma}\label{lema1}
 Let $\u$ be a holomorphy type with constants as in (\ref{constantes}),  $k\in\mathbb N$, $w_{1},\dots,w_{h}\in E$ and $k_1,\dots,k_h\in\mathbb N$ such that $k_1+\dots+k_h=k$. Then $\delta_{(w_{1}^{k_1},\dots,w_{h}^{k_h})}\circ\frac{d^kf}{k!}$ belongs to $\hbu(X)$.

If $\u$ is $AB$-closed and regular at $E$ then the same holds for any $w\in E''$.
\end{lemma}
\begin{proof}
 We prove the case $w\in E''$. The other case is similar.  By \cite[$\S$10 Proposition 2]{Nac69}, $d^kf\in H(X,\u_k(E))$. %(see for example \cite[47.H]{Muj86}). %\cite[Teo 7.17]{Muj86}
Since $\varphi=\delta_{(w_{1}^{k_1},\dots,w_{h}^{k_h})}$ is a continuous linear form on $\u_k(E)$, $\varphi\circ d^kf$ is in $H(X)$.

Let $B_s(x_0)\subset X$ and denote $\frac{d^mf}{m!}(x_0)$ by $Q_m$. Then, for $y\in B_s(x_0)$ we have, by \cite[p.41 (1)]{Nac69}, $$\frac{d^kf}{k!}(y)= \sum_{m\ge k}\binom{m}{k}\big(Q_m\big)_{(q(y)-q(x_0))^{m-k}}.$$ 
This series is absolutely convergent in $\u_k(E)$. Indeed, for $\delta,\varepsilon>0$ such that $(1+\varepsilon)(\delta+\|q(y)-q(x_0)\|)<s$ and using Remark \ref{stirling} we have,
\begin{eqnarray*}\displaystyle
      \sum_{j\ge0}\delta^j\sum_{m\ge j}\binom{m}{j}\|\big(Q_m\big)_{(q(y)-q(x_0))^{m-j}}\|_{\u_j(E)} & \displaystyle\le & c_\varepsilon\sum_{j\ge0}\delta^j\sum_{m\ge j}\binom{m}{j}(1+\varepsilon)^m\|Q_m\|_{\u_m(E)}\|q(y)-q(x_0)\|^{m-j} \\
      &=& c_\varepsilon p^{x_0}_{(1+\varepsilon)(\delta+\|q(y)-q(x_0)\|)}(f)<\infty.
  \end{eqnarray*}
Then $\varphi\circ\frac{d^kf}{k!}(y)=\sum_{m\ge k}\binom{m}{k}\varphi\circ\big(Q_m\big)_{(q(y)-q(x_0))^{m-k}}$
%=\sum_{m\ge k}\binom{m}{k}\big(Q_m\big)_{w^{k}}(q(y)-q(x_0)),
 and therefore,
$$
\frac{d^{m-k}\big(\varphi\circ\frac{d^kf}{k!}\big)}{(m-k)!}(x_0)=x\mapsto\binom{m}{k}\varphi\circ\big(Q_m\big)_{x^{m-k}}.
%=\binom{m}{k}\big(Q_m\big)_{w^{k}}.
$$
By the above lemma, the differentials of $\varphi\circ\frac{d^kf}{k!}$ are in $\u$.

Let $A$ be an open $X$-bounded  set and $\alpha<d_X(A)$. Let $B_s(x_0)\subset A$.
Then, by (\ref{eq stirling}) and Lemma \ref{lema0}, we have
\begin{eqnarray*}
\alpha^kp_s^{x_0}(\varphi\circ\frac{d^kf}{k!}) & \displaystyle\le & \alpha^k\sum_{m\ge k}s^{m-k}\big\|\frac{d^{m-k}\big(\varphi\circ\frac{d^kf}{k!}\big)}{(m-k)!}(x_0)\big\|_{\u_{m-k}(E)} \\
& \le & e^{h+1}(k_1\dots k_h)^{\frac12}\|w_1\|^{k_1}\dots\|w_h\|^{k_h}\alpha^k\sum_{m\ge k}\binom{m}{k}s^{m-k}\sqrt{\frac{m-k}{m}}\|Q_m\|_{\u_m(E)} \\ 
& \le & C_k\alpha^k\|w_1\|^{k_1}\dots\|w_h\|^{k_h}\sum_{m\ge k}\binom{m}{k}s^{m-k}\|Q_m\|_{\u_m(E)} \\
& \le & C_k\|w_1\|^{k_1}\dots\|w_h\|^{k_h}\sum_{j=0}^\infty\alpha^j\sum_{m\ge j}\binom{m}{j}s^{m-j}\|Q_m\|_{\u_m(E)} \\
% & \le & C_k\|w_1\|^{k_1}\dots\|w_h\|^{k_h}\sum_{m=0}^\infty(s+\alpha)^m\|Q_m\|_{\u_m(E)} \\
& = & C_k\|w_1\|^{k_1}\dots\|w_h\|^{k_h}p_{s+\alpha}^{x_0}(f)\le C_k\|w_1\|^{k_1}\dots\|w_h\|^{k_h}p_{\tilde A}(f),
\end{eqnarray*}
where $\tilde A=\bigcup_{x\in A}B_\alpha(x)$ is an open $X$-bounded set (note that by Lemma \ref{lema B_s+delta}, $B_{s+\alpha}(x_0)$ exists and is contained in $\tilde A$). Therefore $p_A(f)<\infty$ and thus $\varphi\circ\frac{d^kf}{k!}$ is in $\hbu(X)$.
\end{proof}
The following corollary states that, for $h=1$ or $h=2$ in the previous lemma, we obtain bounds which are independet of $k$.
\begin{corollary}\label{coro lema1}
Let $\u$ be a holomorphy type with constants as in (\ref{constantes}) and $f\in\hbu(X)$. Let $A$ be an open $X$-bounded set and $\alpha<\rho<d_X(A)$. Define the open $X$-bounded set $\tilde A=\bigcup_{x\in A}B_\rho(x)$. Then there exist a positive constant $C$
depending only on $\alpha$ and $\rho$, such that:\\
$(i)$ for each $k\in\mathbb N$ and $w\in E$, $\delta_w\circ\frac{d^kf}{k!}$ belongs to $\hbu(X)$ and
$$
\alpha^k p_A(\delta_w\circ\frac{d^kf}{k!})\le C\|w\|^kp_{\tilde A}(f).
$$
$(ii)$ for each $l\le k\in\mathbb N$ and $v,w\in E$, $\delta_{(v^{k-l},w^l)}\circ\frac{d^kf}{k!}$ belongs to $\hbu(X)$ and
$$
\alpha^k p_A(\delta_{(v^{k-l},w^l)}\circ\frac{d^kf}{k!})\le C\|v\|^{k-l}\|w\|^lp_{\tilde A}(f).
$$
If $\u$ is $AB$-closed and regular at $E$ then the above statements hold for any $v,w\in E''$.
%
% the series $\sum_k \alpha^k\varphi_k\circ\frac{d^kf}{k!}$ is absolutely convergent in the Banach space $\overline{\hbu(X)}^{p_{A}}$ and $$p_A(\sum_k\alpha^k\varphi_k\circ\frac{d^kf}{k!})\le \sum_k\alpha^k p_A(\varphi_k\circ\frac{d^kf}{k!})\le \frac{c_A}{d_X(A)-\alpha} p_{\tilde A}(f),$$ where $\tilde A=\bigcup_{x\in A}B(x,\frac{2d_X(A)+\alpha}{3})$.\textcolor{red}{$c_A$ tb depende de $d_X(A)-\alpha$, de hecho basta tomar $\varepsilon<(d_X(A)-\alpha)/3(d_X(A)-diam(A))$ en la demo}
\end{corollary}
\begin{proof}
 We prove $(ii)$ for $v,w\in E''$. Let $\varepsilon>0$ such that $\alpha(1+\varepsilon)<\rho$ and let $B_s(x_0)\subset A$.
By the bound obtained in Lemma \ref{lema1} for $h=2$ and $\varphi=\delta_{(v^{k-l},w^l)}$, we have
\begin{eqnarray*}
\alpha^kp_s^{x_0}(\varphi\circ\frac{d^kf}{k!})
& \le & e^3((k-l)l)^{\frac12}\|w\|^{l}\|v\|^{k-l}\alpha^k\sum_{m\ge k}\binom{m}{k}s^{m-k}\sqrt{\frac{m-k}{m}}\|Q_m\|_{\u_m(E)} \\
& \le & c_\varepsilon (1+\varepsilon)^{k}\|w\|^{l}\|v\|^{k-l}\alpha^k\sum_{m\ge k}\binom{m}{k}s^{m-k}\|Q_m\|_{\u_m(E)} \\
& \le & c_\varepsilon\|w\|^{l}\|v\|^{k-l} \sum_{j=0}^\infty(1+\varepsilon)^j\alpha^j\sum_{m\ge j}\binom{m}{j}s^{m-j}\|Q_m\|_{\u_m(E)} \\
& = & c_\varepsilon\|w\|^{l}\|v\|^{k-l}p_{s+(1+\varepsilon)\alpha}^{x_0}(f)\le c_\varepsilon\|w\|^{l}\|v\|^{k-l}p_{\tilde A}(f),
\end{eqnarray*}
where $c_\varepsilon$ is chosen so that $e^3j\le c_\varepsilon (1+\varepsilon)^{j}$ for every $j\in\mathbb N$.
\end{proof}

We proved in the previous section that if $\u$ is multiplicative with constants as in $(\ref{constantes})$, then $\hbu(X)$ is a Fr\'echet algebra. We denote by $\mbu(X)$ its spectrum, that is, the set of non-zero multiplicative and continuous linear functionals on $\hbu(X)$.
Note that evaluations at
points of $X$ are in $\mbu(X)$.
Following \cite{AroColGam91,DinVen04}, we define $\pi:\mbu(X)\to E''$ by $\pi(\varphi)(x')=\varphi(x'\circ q)$. Also, for $\varphi\in M_{b\u}(X)$ and $A$ an open $X$-bounded set, we will write $\varphi\prec A$ whenever there is some $c>0$ such that $\varphi(f)\le cp_A(f)$ for every $f\in\hbu(X)$.

\begin{lemma}\label{hbudex Riemann lema1}
Let $\u$ be a multiplicative holomorphy type with constants as in (\ref{constantes}). Let $\varphi\in M_{b\u}(X)$ and $A$ an open $X$-bounded set such that $\varphi\prec A$. Let $w\in E$ with $\|w\|<d_X(A)$. Then $\varphi^w$
belongs to $M_{b\u}(X)$, where $$\varphi^w (f):=\sum_{n=0}^\infty\varphi(\frac{d^nf(\cdot)}{n!}(w)).$$
Moreover, $\pi(\varphi^w)=\pi(\varphi)+w$. If $\u$ is also $AB$-closed and regular at $E$ then the above statements hold for $w\in E''$ with $\|w\|<d_X(A)$, where,
$$
\varphi^w (f):=\sum_{n=0}^\infty\varphi(AB(\frac{d^nf(\cdot)}{n!})(w)).
$$
\end{lemma}
\begin{proof}
We prove the case $w\in E''$.
 Suppose that $B_{s}(x_0)\subset A$.
Let  $\|w\|<\alpha<\rho<d_X(A)$. We can take $\varepsilon>0$,
such that
$(1+\varepsilon)\alpha<\rho$.
By Corollary \ref{coro lema1}, for the open
$X$-bounded set $\tilde A=\bigcup_{x\in A}B_\rho(x)$, we have
\begin{equation}\label{equat}
\alpha^kp_A\Big(AB(\frac{d^kf(\cdot)}{k!})(w)\Big)\le c_\varepsilon
\|w\|^kp_{\tilde A}(f).
\end{equation}
Thus,  the series
$\sum_kp_A(AB(\frac{d^kf(\cdot)}{k!})(w))$ is convergent. Then,
$$
\sum_{k=0}^\infty|\varphi(AB(\frac{d^kf(\cdot)}{k!})(w))|\le c\sum_{k=0}^\infty p_A(AB(\frac{d^kf(\cdot)}{k!})(w))\le \frac{\alpha c c_\varepsilon}{\alpha-\|w\|} p_{\tilde A}(f)<\infty.
$$
  Therefore $\varphi^w$ is continuous and $\varphi^w\prec \tilde A$.
The multiplicativity of $\varphi^w$ and the last assertion follow as
in \cite[p.551]{AroGalGarMae96}.
\end{proof}

In the case of entire functions $\varphi^w$ may be defined translating functions by $w$, see \cite{CarDimMur,Din99}. We show next that, this is also true for arbitrary Riemann domains when we complete $\hbu(X)$ with respect to the topology given by the norm $p_A$, if $\varphi\prec A$. Given an open $X$-bounded set $A$ and $w\in E''$ with $\|w\|<d_X(A)$, we define $\tilde\tau_w(f)$ on $A$ as $\tilde\tau_w(f)(x)=AB(f\circ (q|_{B_x})^{-1})(q(x)+w),$ where $B_x$ denotes the ball $B_{d_X(x)}(x)$.
\begin{lemma}\label{lema tau_w}
Let $\u$ be a multiplicative holomorphy type with constants as in (\ref{constantes}) which is regular at $E$ and $AB$-closed. Let $A$ be an open $X$-bounded set, $\varphi\in M_{b\u}(X)$ such that $\varphi\prec A$ and $w\in E''$ with $\|w\|<d_X(A)$. Then:
\begin{itemize}
\item[($a$)] the series $\sum_{n=0}^\infty AB(\frac{d^nf(\cdot)}{n!})(w)$ converges in  $\overline{\hbu(X)}^{p_{A}}$ to $\tilde\tau_w(f)$ and the application $\tilde\tau_w:\hbu(X)\to\overline{\hbu(X)}^{p_{A}}$ is continuous,

\item[($b$)] $\varphi$ may be extended to $\overline{\hbu(X)}^{p_{A}}$ and $\varphi^w(f)=\varphi(\tilde\tau_wf)$.
\end{itemize}
\end{lemma}
\begin{proof}
($a$)
We have already proved in  Lemma \ref{hbudex Riemann lema1} that the series
$\sum_np_A(AB(\frac{d^nf(\cdot)}{n!})(w))$ is convergent and that $p_A(\sum_n
AB(\frac{d^nf(\cdot)}{n!})(w))\le\frac{\alpha
cc_\varepsilon}{\alpha-\|w\|} p_{\tilde A}(f)$.

The equality
$\tilde\tau_w(f)(x)=\sum_{n=0}^\infty AB(\frac{d^nf(x)}{n!})(w)$ is
clear for each $x\in A$ since the Taylor series of $f$ at $x$
converges absolutely on $B_{r}(x)$, for each $r<d_X(A)$.

($b$) The first assertion is immediate since $\varphi(f)\le cp_{A}(f)$ for every $f\in\hbu(X)$.
The second assertion, is a consequence of the equality $\tilde\tau_w(f)=\sum_{n=0}^\infty AB(\frac{d^nf(\cdot)}{n!})(w)$ as function in $\overline{\hbu(X)}^{p_{A}}$, the continuity of $\varphi$ with respect to the norm $p_A$ and the definition of $\varphi^w$.
\end{proof}
\begin{lemma}\label{lema phi^v+w}
Let $\u$ be a multiplicative holomorphy type with constants as in (\ref{constantes}) which is regular at $E$ and $AB$-closed. Let $A$ be an open $X$-bounded set, $\varphi\in M_{b\u}(X)$ such that $\varphi\prec A$ and $v,w\in E''$ with
$\|v\|+\|w\|< d_X(A)$. Then:
\begin{itemize}
\item[($a$)]  $\tilde\tau_w\tilde\tau_vf=\tilde\tau_{w+v}f$ for every $f\in\hbu(X)$, 

\item[($b$)] $(\varphi^w)^v$ is a well defined character in $\mbu(X)$
and $(\varphi^w)^v=\varphi^{w+v}$.
\end{itemize}
\end{lemma}

\begin{proof}
($a$) We must prove that $\tilde\tau_w\tilde\tau_vf(x)=\tilde\tau_{w+v}f(x)$ for $x\in A$. Write $g_n$ for $AB(\frac{d^nf(\cdot)}{n!})(v)$. Then \begin{equation}\label{tauw tauv}
\tilde\tau_w\tilde\tau_v(f)=\sum_n\tilde\tau_w(g_n)=\sum_n\sum_kAB(\frac{d^kg_n(\cdot)}{k!})(w).
\end{equation}
Since $\u$ is regular at $E$, we may proceed as in \cite[p.552]{AroGalGarMae96} to show that
$$
\frac{d^kg_n(\cdot)}{k!}=\binom{k+n}{n}AB\Big(\frac{d^{k+n}f(\cdot)}{(k+n)!}\Big)_{v^n}.
$$
Thus again by regularity,
$$
AB\Big(\frac{d^kg_n(\cdot)}{k!}\Big)(w)=\binom{k+n}{n}AB\Big(\frac{d^{k+n}f(\cdot)}{(k+n)!}\Big)^\vee(v^n,w^k)=\binom{k+n}{n}\delta_{(v^n,w^k)}\circ\frac{d^{k+n}f}{(k+n)!}.
$$
Let $\|v\|+\|w\|<\tilde\alpha<\tilde\rho<d_X(A)$. By Corollary \ref{coro lema1}, if $A^\sharp=\bigcup_{x\in A}B_{\tilde\rho}(x)$, then there exists a constant $C>0$  such that,
\begin{eqnarray*}
\sum_{n\ge 0}\sum_{k\ge 0}\binom{k+n}{n}p_A\big(\delta_{(v^n,w^k)}\circ\frac{d^{k+n}f}{(k+n)!}\big) &\le&  C\sum_{n\ge 0}\sum_{k\ge 0}\binom{k+n}{n}\frac{\|v\|^n\|w\|^{k}}{\tilde\alpha^{n+k}}p_{A^\sharp}(f) \\
& = & C p_{A^\sharp}(f)\sum_{m\ge0}^\infty\Big(\frac{\|v\|+\|w\|}{\tilde\alpha}\Big)^m <\infty.
  \end{eqnarray*}
Therefore we may reverse the order of summation in (\ref{tauw tauv}) to obtain
\begin{eqnarray*}
\tilde\tau_w\tilde\tau_v(f) &=& \sum_{n\ge 0}\sum_{l\ge n}\binom{l}{n}AB\big(\frac{d^{l}f(\cdot)}{l!}\big)^\vee(v^n,w^{l-n})
= \sum_{l\ge 0}\sum_{n=0}^l\binom{l}{n}AB\big(\frac{d^{l}f(\cdot)}{l!}\big)^\vee(v^n,w^{l-n}) \\
&=& \sum_{l\ge 0}AB\big(\frac{d^{l}f(\cdot)}{l!}\big)(v+w) = \tilde\tau_{v+w}(f).
\end{eqnarray*}

($b$)
We continue using the notation of part ($a$).
First note that $\tilde\alpha-\|v\|$ (resp. $\tilde\rho-\|v\|$) may play the role of $\alpha$ (resp. $\rho$) in the proof of Lemma \ref{hbudex Riemann lema1}, thus
if $\tilde A=\bigcup_{x\in A}B_{\tilde\rho-\|v\|}(x)$,
then  $\varphi^w$ and $\tilde\tau_w$ may be continuously extended to $\Big(\overline{\hbu(X)}^{p_{\tilde A}},p_{\tilde
A}\Big)$. Moreover, the formula
$\varphi^w(f)=\varphi(\tilde\tau_wf)$ holds for every $f$ in
$\overline{\hbu(X)}^{p_{\tilde A}}$.

Second, since $\varphi^w\prec\tilde A$ and $d_X(\tilde A)\ge
d_X(A)-(\tilde\rho-\|v\|)>d_X(A)-\|w\|>\|v\|$, then by of Lemma \ref{hbudex Riemann lema1},  $(\varphi^w)^v$ is well defined
and by part ($a$) of Lemma \ref{lema tau_w}, $\tilde\tau_v:\hbu(X)\to\overline{\hbu(X)}^{p_{\tilde A}}$ is
continuous.

Therefore, for every $f\in\hbu(X)$, we have that
$(\varphi^w)^v(f)=\varphi^w(\tilde\tau_vf)=\varphi(\tilde\tau_w\tilde\tau_vf)=\varphi(\tilde\tau_{w+v}f)=\varphi^{w+v}(f)$.
\end{proof}
The equality  $(\varphi^w)^v=\varphi^{w+v}$ in the above lemma is the key property to show Theorem \ref{mbu dom riemann}. 
Indeed, once this equality is proved, the  rest of the proof can be almost entirely adapted from \cite[Theorem 2.2 and Corollary 2.4]{AroGalGarMae96}. We just point out the only difference.
\begin{proof}[Proof of Theorem \ref{mbu dom riemann}]
 For $\varphi\in\mbu(X)$, $\varphi\prec A$ and $0<\varepsilon<d_X(A)$, define $V_{\varphi,\varepsilon}=\{\varphi^w:\,w\in E'',\, \|w\|<\varepsilon\}$.  Then the collection
$\{V_{\varphi,\varepsilon}:\varphi\in M_{b\u}(X),\, \varepsilon>0\}$ define a basis for a
Hausdorff topology in $M_{b\u}(X)$. The fact that it is a basis of a topology follows as in \cite[Theorem 2.2]{AroGalGarMae96}. We prove that it is Hausdorff.  Let $\varphi\ne\psi\in\mbu(X)$ and suppose that $\pi(\varphi)\ne\pi(\psi)$. Let $A,D$ be open $X$-bounded sets such that $\varphi\prec A$ and $\psi\prec D$, and take $r<\min\{d_X(A),d_X(D)\}/2$. We claim that $V_{\varphi,r}\cap V_{\psi,r}=\emptyset$. Indeed, if $\|v\|,\|w\|<r$ are such that $\varphi^w=\psi^v$, then $\pi(\varphi)+w=\pi(\psi)+v$ and thus $v=w$. Moreover, by Lemma \ref{lema phi^v+w} $(b)$, $\varphi=(\varphi^v)^{(-v)}=(\psi^v)^{(-v)}=\psi.$
The case $\pi(\varphi)=\pi(\psi)$ follows as in \cite[Theorem 2.2]{AroGalGarMae96}. Now, we may finish the proof of the theorem proceeding as in \cite[Corollary 2.4]{AroGalGarMae96}.
\end{proof}

\section{Holomorphic extensions}\label{section holo extensions}

In this section and in the next one we are concerned with analytic continuation.
We show first that the canonical extensions to the spectrum are holomorphic and then we characterize the $\hbu$-envelope of holomorphy of a Riemann domain in terms of the spectrum.
\begin{proposition}\label{gelfand holo}
Let $\u$ be a multiplicative holomorphy type with constants as in (\ref{constantes}) which is regular at $E$ and $AB$-closed. For each $f\in\hbu(X)$, its Gelfand transform $\tilde f$ is holomorphic on $\mbu(X)$.
\end{proposition}
\begin{proof}
Let $\varphi\in\mbu(X)$, $A$ an open $X$-bounded set such that $\varphi\prec A$ and $r<d_X(A)$. We prove that $\tilde f$ is holomorphic on $V_{\varphi,r}$, or equivalently that $\tilde f\circ\big(\pi|_{V_{\varphi,r}}\big)^{-1}$ is holomorphic on $\pi(V_{\varphi,r})=B_{E''}(\pi(\varphi),r)$. It suffices to show that it is uniform limit of polynomials on $rB_{E''}(\pi(\varphi))$. Note that for $\|w\|<r$,
$$
\tilde f\circ\big(\pi|_{V_{\varphi,r}}\big)^{-1}(\pi(\varphi)+w)=\tilde f(\varphi^w)=\varphi^w(f)=\sum_{k=0}^\infty\varphi(\delta_w\circ\frac{d^kf}{k!}).
$$
By Lemma \ref{lema1}, for $w_1,\dots,w_k\in E''$, $AB\big(\frac{d^kf}{k!}\big)^\vee(w_1,\dots,w_k)=\delta_{(w_1,\dots,w_k)}\circ\frac{d^kf}{k!}$ belongs to $\hbu(X)$ and clearly, $AB\big(\frac{d^kf}{k!}\big)^\vee$ is clearly $k$-linear. Thus $w\mapsto\delta_w\circ\frac{d^kf}{k!}$ is in $\p_a^k(E'',\hbu(X))$ (the space of algebraic $k$-homogeneous polynomials). It is also continuous since, by Corollary \ref{coro lema1}, for each open $X$-bounded  set $D$, and $\beta<d_X(D)$, there exists $C>0$ and an open $X$-bounded  set $\tilde D$ (which do not dependend on $k$) such that $\beta^kp_D(\delta_w\circ\frac{d^kf}{k!})\le C\|w\|^kp_{\tilde D}(f)<\infty.$

Therefore, if $Q_k(w)=\varphi(\delta_w\circ\frac{d^kf}{k!})$ then $Q_k$ is in $\p^k(E'')$. Now, for $\|w\|<r<\alpha<d_X(A)$ and using again Corollary \ref{coro lema1},
\begin{eqnarray*}
  \sup_{w\in rB_{E''}}\Big| \varphi^w(f)-\sum_{n=0}^mQ_n(w)\Big| &=& \sup_{w\in rB_{E''}}\Big| \sum_{k=m+1}^\infty \varphi(\delta_w\circ\frac{d^kf}{k!})\Big|\\
&\le& c\sup_{w\in rB_{E''}}\Big| \sum_{k=m+1}^\infty p_A(\delta_w\circ\frac{d^kf}{k!})\Big|\\
&\le& cC \sum_{k=m+1}^\infty \big(\frac{r}{\alpha}\big)^kp_{\tilde A}(f)\longrightarrow0\textrm{ as }m\to\infty.
%&\le & cC\big(\frac{r}{\alpha}\big)^m\frac{\alpha}{\alpha-r} p_{\tilde A}(f)\longrightarrow0,
\end{eqnarray*}
\end{proof}
\begin{remark}\rm\label{remark ext son de hu}
 We know that the canonical extensions to $\mbu(X)$ need not be in $\hbu(\mbu(X))$ (see \cite[Example 2.8]{CarMur} and \cite[Proposition 4.3.22]{Mur10}). We do not know whether these extensions belong to $\hu(\mbu(X))$.
When $\u$ is weakly differentiable (see Section \ref{section type u extensions}) and $f$ is a polynomial, then it is possible to show that its extension to $\mbu(X)$ is of type $\u$.
\end{remark}
\begin{definition}\label{def F-ext}\rm
Let $\mathcal F\subset H(X)$ and let $(Z,p)$ be another Riemann domain over $E$.
An \textbf{$\mathcal F$-extension} is a morphism $\tau:X\to Z$ such that for each $f\in\mathcal F$ there exists a unique function $\tilde f$ holomorphic on $Z$, such that $\tilde f\circ\tau=f$. If $\mathcal F=H(X)$, we call it a holomorphic extension of $X$.
\end{definition}
\begin{corollary}\label{mbu dom holo}
 Let $\u$ be a multiplicative holomorphy type with constants as in (\ref{constantes}) which is regular at $E$ and $AB$-closed. Then $\mbu(X)$ is a domain of holomorphy, that is, any holomorphic extension of $\mbu(X)$ is an isomorphism.
\end{corollary}
\begin{proof}
We may follow the steps of \cite[Proposition 2.4]{DinVen04}. By \cite[Theorem 52.6]{Muj86} it sufices to prove that $\mbu(X)$ is holomorphically separated
and that for each sequence $\{\varphi_j\}$ in $\mbu(X)$ such that $d_{\mbu(X)}(\varphi_j)\to0$,
there exists a function $F$ in $H(\mbu(X))$ such that $\sup_j|F(\varphi_j)|=\infty$.
Since $\mbu(X)$ is separated by $\hbu(X)$, it is holomorphically separated by the above proposition. If $\{\varphi_j\}\subset\mbu(X)$ is such that
$\sup_j |F(\varphi_j )|<\infty$ for all $F\in H(\mbu(X))$, then if $\tau(f):=\sup_j|\varphi_j(f)|$, $\tau$ defines a seminorm on $\hbu(X)$. Thus the set $V=\{f\in\hbu(X):\, \tau(f)\le 1\}$ is absolutely convex and absorvent. It is also closed because $V$ is the intersection of the closed sets $\{f\in\hbu(X):\, |\varphi_j(f)|\le 1\}$. Since $\hbu(X)$ is a barreled space, $V$ is a neighbourhood of 0 and thus $\tau$ is continuous. Therefore, there are an $X$-bounded set $D$ and a constant $c>0$ such that $\tau(f)\le cp_D(f)$ for every $f\in\hbu(X)$, which implies that $\varphi_j\prec D$ for every $j\in\mathbb N$. By Lemma \ref{hbudex Riemann lema1}, $d_{\mbu(X)}(\varphi_j)\ge d_X(D)$.
\end{proof}
Let $(X,q)$ be a connected Riemann domain over $E$. The envelope of holomorphy of $X$ is an extension which is maximal in the sense that any other extension factorizes through it.
\begin{definition}\label{def hbu envelope}\rm
The \textbf{$\hbu$-envelope} of $X$ is a Riemann domain $\mathcal E_{b\u}(X)$ and an $\hbu$-extension morphism $\tau:X\to\mathcal E_{b\u}(X)$ such that if $\nu:X\to Z$ is another $\hbu$-extension, then there exists a morphism $\mu:Z\to\mathcal E_{b\u}(X)$ such that $\nu\circ\mu=\tau$.
\end{definition}
 In \cite{Hir72}, Hirschowitz proved, using germs of analytic functions, the existence of $\mathcal E_{b\u}(X)$  (in a more general framework) and asked whether the extended functions $\tilde f$ are also of type $\u$
on $\mathcal E_{b\u}(X)$, \cite[p. 290]{Hir72}. We will give a partial positive answer to this question in the next section (Corollary \ref{extensiones al envelope}).

We now characterize the $\hbu$-envelope of holomorphy of $X$ in terms of the spectrum of $H_{b\u}(X)$. We sketch the proof which is an adaptation of \cite[Theorem 1.2]{CarMur}.
First note that the conditions that $\u$ be $AB$-closed and regular at $E$ were used in Theorem \ref{mbu dom riemann} only  to deal with Aron-Berner extensions. Thus, it is not difficult to show the following. 
\begin{lemma}
Let $(X,p)$ be a Riemann domain spread over a Banach space $E$ and let $\u$ be a multiplicative holomorphy type with constants as in (\ref{constantes}). Then $(\pi^{-1}(E),\pi)\subset(M_{b\u}(X),\pi)$ is a Riemann domain spread over $E$.
\end{lemma}
\begin{proposition}
Let $(X,p)$ be a connected Riemann domain spread over a Banach space $E$, let $\u$ be a multiplicative holomorphy type with constants as in (\ref{constantes}) and let $Y$ be the connected component of $\pi^{-1}(E)\subset\mbu(X)$ which intersects $\delta(X)$. Then $\delta:(X,p)\to(Y,\pi)$, $\delta(x)=\delta_x$ is the $\hbu$-envelope of $X$.
\end{proposition}
\begin{proof}
Let $\sigma:X\to \mathcal E_{b\u}(X)$ be the $\hbu$-extension from $X$ to the $\hbu$-envelope of $X$.
By Proposition \ref{gelfand holo},  $\delta:X\to Y$ is an $\hbu$-extension. Moreover, for each point $y\in Y$, the evaluation $\delta_y:\hbu(X)\to\mathbb C$, $\delta_y(f)=\tilde f(y)$ is continuous.
Then there is a morphism $\nu:Y\to \mathcal E_{b\u}(X)$ such that $\sigma=\nu\circ\delta$.

We show that $\nu$ is an isomorphism. $\nu(Y)$ is open in $\mathcal E_{b\u}(X)$ because $\nu$ is a morphism.

Let us see that $\nu(Y)$ is closed in $\mathcal E_{b\u}(X)$. Suppose that $y\in\overline{\nu(Y)}\setminus\nu(Y)$. Let $W_n=\{\varphi\in Y:\, \varphi\prec X_n\}$, where $X_n=\{x\in X:\, \|p(x)\|\le n,\, d_X(x)\ge\frac1{n}\}$. Then by Lemma \ref{hbudex Riemann lema1}, $d_{Y}(W_n)\ge\frac1n$. Therefore we can get a subsequence of integers $(n_k)_k$ and a sequence $(y_k)_k\subset Y$ such that $y_k\in W_{n_{k+1}}\setminus W_{n_k}$ and $y_k\to y$. Thus there are functions $f_k\in \hbu(X)$ such that $p_{X_{n_k}}(f_k)<\frac1{2^k}$ and $|\tilde f_k(y_k)|>k+\sum_{j=1}^{k-1}|\tilde f_j(y_k)|$. Then the series $\sum_{j=1}^{\infty}f_j$ converges to a function $f\in \hbu(X)$ and
moreover $\Big|\Big(\sum_{j=1}^{\infty}f_j\Big)^\sim(y_k)\Big| =\Big|\sum_{j=1}^{\infty}\tilde f_j(y_k)\Big|$ because $\delta_{y_k}$ is $\hbu(X)$-continuous. Therefore
$$
|\tilde f(y_k)| =\Big|\sum_{j=1}^{\infty}\tilde f_j(y_k)\Big| \ge|\tilde f_k(y_k)|-\Big|\sum_{j=1}^{k-1}\tilde f_j(y_k)\Big|-\Big|\sum_{j=k+1}^{\infty}\tilde f_j(y_k)\Big|>k-1,
$$
so we have that $|\tilde f(y_k)|\to\infty$  and then $f$ cannot be extended to $y$. This is a contradiction since $y$ belongs to the $\hbu$-envelope of $X$, $\mathcal E_{b\u}(X)$. Thus $\nu(Y)$ is closed in $\mathcal E_{b\u}(X)$.
\end{proof}

\section{Type $\u$ extensions}\label{section type u extensions}

It is also natural to consider extensions where the extended functions are not only holomorphic but also of type $\u$. As mentioned in Remark \ref{remark ext son de hu}, we cannot expect in general that the extensions be of  $\u$-bounded type, since even for the current type, the extension of a bounded type function to the $H_b$-envelope of holomorphy may fail to be of bounded type. We may ask instead if, at least, they are in $\hu$.
\begin{definition}\label{hbu-hu-extension}\rm
A Riemann domain morphism $\tau:(X,q)\to(Y,\tilde q)$ is an \textbf{$\hbu$-$\hu$-extension} if for each $f\in\hbu(X)$ there exists a unique $\tilde f\in\hu(Y)$ such that $\tilde f\circ\tau=f$.
\end{definition}
For the current type, $\hu$ is the space of all holomorphic functions, thus in this case, every extension of a function in $\hbu=H_b$ belongs to $\hu=H$. On the other hand, Dineen found (see \cite[Example 11]{Din71(holomorphy-types)}) an entire function $f$ of bounded type  that has nuclear radius of convergence  $r>0$ and  such that there exists $x\in E$ for which $d^2f(x)\notin \p_N^2(E)$. This means that $f$ belongs to $H_{bN}(rB_E)$ and it extends to an entire function in $H_b(E)$, but this extension is not in $H_N(E)$. Thus, the extension of a single function in $\hbu$ need not be of type $\u$. 
In this section we show, under the additional hypothesis of weak differentiability, that when all functions in $\hbu$ are extended simultaneously  (that is, when one deals with $\hbu$-extensions), the extended functions are of type $\u$.
\begin{remark}\rm
There is also a corresponding notion of $\hbu$-$\hu$-envelope of holomorphy.
This was considered by Moraes in \cite{Mor81}, where she proved, using germs of analytic functions, that the $\hbu$-$\hu$-envelope of holomorphy always exists. Thus, a positive answer to the question of Hirschowitz whether the extensions to the $\hbu$-envelope are of type $\u$ is equivalent to the coincidence of 
 the $\hbu$-envelope with the $\hbu$-$\hu$-envelope.
This will be proved in Corollary \ref{extensiones al envelope} for weakly differentiable sequences.
\end{remark}
\begin{definition}\rm
Let $\u$ be a sequence of polynomial ideals and let $E$ be
a Banach space. We say that $\u$ is \textbf{weakly differentiable} at $E$ if there
exist constants $c_{k,l}>0$ such that, for  $l < k$,  $P\in\u_k(E)$ and $\varphi\in\u_{k-l}(E)'$,  the mapping
$x\mapsto\varphi(P_{x^l})$ belongs to $\u_l(E)$ and
$$\Big\|x\mapsto\varphi\big(P_{x^l}\big)\Big\|_{\u_l(E)} \le c_{k,l}\|\varphi\|_{\u_{k-l}(E)'}\|P\|_{\u_k(E)}.$$
\end{definition}
\begin{remark}\label{w-dif dual a mult}\rm
Weak differentiability, a condition which is stronger than being a holomorphy type, was defined in \cite{CarDimMur} and is dual to multiplicativity in the following sense: if $\{\u_k\}_k$ is a weakly differentiable sequence then the sequence of adjoint ideals $\{\u_k^{*}\}_k$ is multiplicative (with the same constants); and if $\{\u_k\}_k$ is multiplicative then the sequence of adjoint ideals $\{\u_k^{*}\}_k$ is weakly differentiable (with the same constants), see \cite[Proposition 3.16]{CarDimMur}.
\end{remark}
All examples appearing in Section \ref{section prelim} but $\mathcal M_r$ and $\mathcal S_p$ were shown in \cite{CarDimMur} to be weakly differentiable sequences. It is not difficult to see that the constants satisfy (\ref{constantes}) in all those cases. We don't know if the sequence of multiple $r$-summing polynomials is weakly differentiable. We prove now that the sequences of Hilbert-Schmidt and Schatten-von Neumann ideals of polynomials are weakly differentiable. Moreover, the duality between multiplicativity and weak differentiability allows us to show they are also multiplicative, and the bound obtained for the norm of the product of Hilbert-Schmidt polynomials improves \cite[Lemma 2.1]{Pet01}.
\begin{example}\label{HS w-diff}
The sequence $\mathcal S_2$ of ideals of Hilbert-Schmidt polynomials is weakly differentiable and multiplicative with constants $c_{k,l}=1$.\rm\begin{proof}                                                                                                                          
 Let $H$ be a Hilbert space with orthonormal basis $(e_i)_i$. Recall that $\mathcal S_2^k(H)$ is the completion of finite type $k$-homogeneous polynomials on $H$ with respect to the norm associated to the inner product
$$
\langle P,Q\rangle_{\mathcal S_2^k(H)}=\sum_{i_1,\dots,i_k}\v P(e_{i_1},\dots,e_{i_k})\overline{\v Q(e_{i_1},\dots,e_{i_k})}.
$$
Let $P\in\p^k(H)$. It is not difficult to deduce (see \cite[Lemma 1]{Dwy71}) that $P$ belongs to $\mathcal S_2^k(H)$ if and only if it is (uniquely) expressed as a limit in the $\mathcal S_2^k(H)$-norm by
\begin{equation}\label{expresion HS}
 P=\sum_{i_1,\dots,i_k}a_{i_1\dots i_k}e'_{i_1}\dots e_{i_k}',
\end{equation}
with symmetric coefficients $a_{i_1\dots i_k}\in\mathbb C$ and
$$
\sum_{i_1,\dots,i_k}|a_{i_1\dots i_k}|^2=\|P\|_{\mathcal S_2^k(H)}^2<\infty.
$$
Let $\varphi\in\mathcal S_2^l(H)'$ and let $Q=\sum_{i_1,\dots,i_l}b_{i_1\dots i_l}e'_{i_1}\dots e_{i_l}'\in\mathcal S_2^l(H)$ be such that $\varphi=\langle \cdot,Q\rangle_{\mathcal S_2^l(H)}$. Then
$$
\langle P_{x^{k-l}},Q\rangle_{\mathcal S_2^l(H)}=\sum_{i_1,\dots,i_l}\big(\sum_{i_{l+1},\dots,i_k}a_{i_1\dots i_k}x_{i_{l+1}}\dots x_{i_k}\big)\overline{b_{i_1\dots i_l}}.
$$
This series is absolutely convergent, indeed
\begin{eqnarray*}
\Big(\hspace{-.1cm}\sum_{i_1,\dots,i_l}\Big(\hspace{-.2cm}\sum_{i_{l+1},\dots,i_k}\hspace{-.3cm}|a_{i_1\dots i_k}x_{i_{l+1}}\dots x_{i_k}|\Big)|b_{i_1\dots i_l}|\Big)^2 & \le & \|Q\|^2_{\mathcal S_2^l(H)}\sum_{i_1,\dots,i_l}\Big(\hspace{-.2cm}\sum_{i_{l+1},\dots,i_k}\hspace{-.3cm}|a_{i_1\dots i_k}x_{i_{l+1}}\dots x_{i_k}|\Big)^2 \\ &\le & \|Q\|^2_{\mathcal S_2^l(H)}\sum_{i_1,\dots,i_l}\Big(\hspace{-.2cm}\sum_{i_{l+1},\dots,i_k}\hspace{-.3cm}|a_{i_1\dots i_k}|^2\Big)\Big(\hspace{-.2cm}\sum_{i_{l+1},\dots,i_k}\hspace{-.3cm}|x_{i_{l+1}}\dots x_{i_k}|^2\Big)  \\
&\le & \|Q\|^2_{\mathcal S_2^l(H)}\|P\|^2_{\mathcal S_2^k(H)}\|x\|^{2(k-l)}.
\end{eqnarray*}
Thus reversing the order of summation we obtain,
$$
x\mapsto\langle P_{x^{k-l}},Q\rangle_{\mathcal S_2^l(H)}=\sum_{i_{l+1},\dots,i_k}\Big(\sum_{i_1,\dots,i_l}a_{i_1\dots i_k}\overline{b_{i_1\dots i_l}}\Big)e_{i_{l+1}}'\dots e_{i_k}'.
$$
Note that this representation is as in (\ref{expresion HS}) and since
$$
 \sum_{i_{l+1},\dots,i_k}\Big|\sum_{i_1,\dots,i_l}a_{i_1\dots i_k}\overline{b_{i_1\dots i_l}}\Big|^2 \le \sum_{i_{l+1},\dots,i_k}\big(\sum_{i_1,\dots,i_l}|a_{i_1\dots i_k}|^2\big)\big(\sum_{i_1,\dots,i_l}|b_{i_1\dots i_l}|^2\big)
 =\|P\|^2_{\mathcal S_2^k(H)}\|Q\|^2_{\mathcal S_2^l(H)},
$$
we conclude that $x\mapsto\langle P_{x^{k-l}},Q\rangle_{\mathcal S_2^l(H)}$ is in $\mathcal S_2^l(H)$ and has $\mathcal S_2^l(H)$-norm $\le\|P\|_{\mathcal S_2^k(H)}\|Q\|_{\mathcal S_2^l(H)}$, that is, $\mathcal S_2$ is weakly differentiable with $c_{k,l}=1$.

Moreover, since the adjoint ideal of $\mathcal S_2^k$ (as a normed ideal of polynomials on Hilbert spaces) is again $\mathcal S_2^k$ and since the sequence of adjoint ideals of a weakly differentiably sequence is multiplicative with the same constants by \cite[Proposition 3.16]{CarDimMur}, we conclude that $\mathcal S_2$ is multiplicative with $c_{k,l}=1$. 
\end{proof}
\end{example}
\begin{example}
The sequence $\mathcal S_p$ of Schatten-von Neumann polynomials.\rm\newline\smallskip
Using the Reiteration theorem for the complex interpolation method, or by duality with Example \ref{schatten multip}, we deduce that $\mathcal S_p$ is weakly differentiable with constants $c_{k,l}=1$ if $1<p<2$ and $$c_{k,l} \le\Big(\frac{(k+l)^{k+l}}{(k+l)!}\frac{k!}{k^k}\frac{l!}{l^l}\Big)^{1-\frac2{p}} $$ for $2<p<\infty$.
\end{example}
We prove now that for weakly differentiable holomorphy types every $\hbu$-extension is an $\hbu$-$\hu$-extension.
\begin{lemma}\label{lema1 cartan-thullen}
Let $\u$ be a  weakly differentiable holomorphy type with constants as in (\ref{constantes}) and $f\in\hbu(X)$. Then for each open $X$-bounded  set $A$ and $\alpha<d_X(A)$, there exist a positive constant $C$ and an open $X$-bounded set $\tilde A$  such that for every $k\in\mathbb N$ and $\varphi\in\u_k(E)'$, $\varphi\circ\frac{d^kf}{k!}$ belongs to $\hbu(X)$ and
$$
\alpha^k p_A(\varphi\circ\frac{d^kf}{k!})\le C\|\varphi\|_{\u_k(E)'}p_{\tilde A}(f).
$$
\end{lemma}
\begin{proof}
Since $\u$ is weakly differentiable, the proof of Lemma \ref{lema1} with the bound obtained in Corollary \ref{coro lema1} works here for any $\varphi\in\u_k(E)'$.
\end{proof}
\begin{proposition}\label{hbu-hu=hbu}
Let $(X,q)$ be a connected Riemann domain spread over a Banach space $E$, let $\u$ be a weakly differentiable holomorphy type with constants as in (\ref{constantes}). Let $e:(X,q)\to(Y,p)$ be an $\hbu$-extension. 
Then for each $y\in Y$ there exists connected open subset $Z$, $e(X)\cup\{y\}\subset Z\subset Y$ such that for every $f\in\hbu(X)$, the extension $\tilde f$ to $Y$ is in $\hbu(Z)$. In particular, $\tilde f$ is of type $\u$ in $Y$.
\end{proposition}
\begin{proof}
Note first that the evaluation at each point $y$ in $Y$ defines a continuous character on $\hbu(X)$, $\delta_y(f):=\tilde f(y)$. Indeed, define the set $V\subset Y$ consisting in all points $y_0$ for which there is an open connected subset $Z$ such that $y_0$ belongs to $Z$ and every point in $Z$ induce a continuous evaluation. Clearly $V$ is an open nonempty set ($e(X)$ is contained in $V$). Moreover, if $(y_n)\subset V$ and $y_n\to y_0$, then since $\delta_{y_n}$ is continuous, the seminorm defined by $f\mapsto\sup_{n\in\mathbb N}|\delta_{y_n}(f)|$ is continuous (the sets $\{f\in\hbu(X):\, \sup_{n\in\mathbb N}|\delta_{y_n}(f)|\le r\}$ are closed, absolutely convex and absorvent and $\hbu(X)$ is a Fr\'echet space, thus they have nonempty interior). Therefore there are an open $X$-bounded  set $A$ and a constant $c>0$ such that $|\tilde f(y_0)|\le\sup_{n\in\mathbb N}|\delta_{y_n}(f)|\le cp_A(f)$ for every $f\in\hbu(X)$. Thus $V$ is closed and since $Y$ is connected, we have $V=Y$.

Take a point $y\in Y$ and let $\gamma:[0,1]\to Y$ be a curve such that $\gamma(0)\in e(X)$ and $\gamma(1)=y$.
By compactness, it follows that there is some open $X$-bounded  set $A$ such that $\delta_{\gamma(t)}\prec A$ for every $t\in[0,1]$, that is, for each $t$, there  exists $c>0$ such that $|\tilde g(\gamma(t))|\le cp_A(g)$ for every $g\in\hbu(X)$.

 Let $I$ denote the set of all $t_0\in[0,1]$ such that there exists a connected open subset $Z\subset Y$
 which contains $e(X)$ and satisfies that $\gamma(t)\in Z$ for every $t\le t_0$ and that $\tilde g|_{Z}$ belongs to $\hbu(Z)$ for every $g\in\hbu(X)$.
 To prove the proposition it is enough to show that $I=[0,1]$. Since $I$ is clearly open, it suffices to prove that if $[0,t_0)\subset I$ then $t_0$ belongs to $I$. Take $t_1<t_0$ such that $\gamma((t_1,t_0])$ is contained in some ball $B$ of center $\gamma(t_1)$ and radius $r<d_X(A)$ in $Y$. Let $Z$ be the subdomain which exists for $t_1$ in the definition of $I$. Note that $e:(X,q)\to(Z,p|_Z)$ is an $\hbu$-extension.

Let $\varphi_k\in \u_k(E)'$ and $f\in\hbu(X)$. By Lemma \ref{lema1 cartan-thullen}, $\varphi_k\circ\frac{d^kf}{k!}\in \hbu(X)$,
and since the extension of $f$ to $Z$, $\tilde f|_{Z}$, belongs to $\hbu(Z)$, we also have that $\varphi_k\circ\frac{d^k\tilde f|_{Z}}{k!}\in\hbu(Z)$.

Moreover, $\varphi_k\circ\frac{d^k\tilde f|_{Z}}{k!}$ is an extension of $\varphi_k\circ\frac{d^kf}{k!}$ to $Z$. Indeed if $x\in X$ and $(V_x,q)$ is a chart at $x$ in $X$ such that $(V_{e(x)},p)=(e(V_{x}),p)$ is a chart at $e(x)$ in $Z$, then
\begin{eqnarray*}
\varphi_k\circ\frac{d^k\tilde f|_{Z}}{k!}(e(x)) & = & \varphi_k\circ\frac{d^k[\tilde f\circ(p|_{V_{e(x)}})^{-1}]}{k!}\big(p(e(x))\big) \overset{(*)}{=}\varphi_k\circ\frac{d^k[f\circ(q|_{V_{x}})^{-1}]}{k!}\big(q(x)\big) \\ & = & \varphi_k\circ\frac{d^kf}{k!}\big(x\big),
\end{eqnarray*}
where $(*)$ is true because $\tilde f\circ(p|_{V_{e(x)}})^{-1}=f\circ(q|_{V_{x}})^{-1}$ since $e$ is an $\hbu$-extension.

 Since $\Big(\varphi_k\circ\frac{d^kf}{k!}\Big)^\sim$ is also an extension of $\varphi_k\circ\frac{d^kf}{k!}$ to $Z$,
 we must have that
\begin{equation}
\Big(\varphi_k\circ\frac{d^kf}{k!}\Big)^\sim
=\varphi_k\circ\frac{d^k\tilde f|_Z}{k!}.
\end{equation}
Therefore for $r<\alpha<d_X(A)$,
\begin{eqnarray*}
\sum_kr^k\big\|\frac{d^k\tilde f|_Z}{k!}(\gamma(t_1))\big\|_{\u_k(E)} & = & \sum_kr^k\sup_{\varphi_k\in B_{\u_k(E)'}}\big|\varphi_k\Big(\frac{d^k\tilde f|_Z}{k!}(\gamma(t_1))\Big)\big|\\ & =& \sum_kr^k\sup_{\varphi_k\in B_{\u_k(E)'}}\big|\Big(\varphi_k\circ\frac{d^kf}{k!}\Big)^\sim(\gamma(t_1))\big| \\
 & \le & \sum_kr^k\sup_{\varphi_k\in B_{\u_k(E)'}}cp_A(\varphi_k\circ\frac{d^kf}{k!}) \\
& \le & cC\frac{\alpha}{\alpha-r} p_{\tilde A}(f) <\infty,
\end{eqnarray*}
where $\tilde A$ and $C$ are, respectively, the $X$-bounded set and the constant given by Lemma \ref{lema1 cartan-thullen}.
Thus $\tilde f$ belongs to $\hbu(Z\cup B)$. Since this holds for every $f\in\hbu(X)$ and $[0,t_0]$ is contained in $Z\cup B$, we conclude that $t_0$ is in $I$.
\end{proof}
As a corollary we have the following partial answer to a question of Hirschowitz (see comments after Definition \ref{def hbu envelope}).
\begin{corollary}\label{extensiones al envelope}
If $\u$ be a weakly differentiable holomorphy type with constants as in (\ref{constantes}), then the extensions to the $\hbu$-envelope of holomorphy are of type $\u$.
\end{corollary}
The following result can be proved as Corollary \ref{mbu dom holo}, but using the above corollary instead of Proposition \ref{gelfand holo}.
\begin{corollary}
 Let $\u$ be a multiplicative and weakly differentiable holomorphy type with constants as in (\ref{constantes}). Then the $\hbu$-envelope of holomorphy is an $\hu$-domain of holomorphy, that is, any $\hu$-extension is an isomorphism.
\end{corollary}

\medskip

The Cartan-Thullen theorem characterizes domains of holomorphy in $\mathbb C^n$ in terms of holomorphic convexity. It was extended for bounded type holomorphic functions on Banach spaces by Dineen \cite{Din71(Cartan-Thullen)}.
Shortly after, Cartan-Thullen type theorems were proved for very general classes of spaces of holomorphic functions on infinite dimensional spaces by Coeur\'e \cite{Coe74}, Schottenloher \cite{Sch72,Sch74} and Matos \cite{Mat72tau,Mat74}. 
%Coeur\'e \cite[Theorems 4.9 and 5.1]{Coe74}
Despite the generality of this theorems\footnote{The natural Fr\'echet spaces considered by Coeur\'e include, by Corollary \ref{coro lema1}, the spaces $\hbu$ and so do the regular classes studied by Schottenloher when the holomorphy type is multiplicative.}, the fact that any holomorphically convex domain is a domain of holomorphy was only proved for spaces of analytic functions which one may associate to the current holomorphy type (spaces of analytic functions which are bounded on certain families of subsets, with the topology of uniform convergence on these subsets). We guess that this may be due to the fact that the concept of holomorphic convexity considered there make use of the seminorms associated to the current type. We propose instead a concept of $\hbu$-convexity which uses the seminorms associated to the corresponding holomorphy type, and show then that a Riemann domain is an $\hbu$-domain of holomorphy if and only if it is $\hbu$-convex.
\begin{definition}\rm
For each open $X$-bounded  set $A$, we define its \textbf{$\hbu(X)$-convex hull} as
$$
\hat A_{\hbu(X)}:=\{x\in X:\;\textrm{ there exists }c>0\textrm{ such that }|f(x)|\le cp_A(f)\textrm{ for every }f\in\hbu(X)\}.
  $$
\end{definition}
\begin{remark}\rm
If the seminorms $p_A$ are submultiplicative (as in the case of $H_b$) then the constant in above definition may be taken $c=1$.
\end{remark}
\begin{definition}\rm
 We say that a Riemann domain ($X,q$) is \textbf{$\hbu$-convex} if for each open $X$-bounded  set $A$, its $\hbu(X)$-convex hull $\hat A_{\hbu(X)}$ is $X$-bounded.
\end{definition}
\begin{definition}\rm
 We say that a Riemann domain ($X,q$) is an \textbf{$\hbu$-domain of holomorphy} (\textbf{$\hbu$-$\hu$-domain of holomorphy}) if each $\hbu$-extension ($\hbu$-$\hu$-extension) morphism is an isomorphism.
\end{definition}
%Para la construccion de $\tilde X$: Muj86 p365, Berg, Naras p109
We are now ready to prove the Cartan-Thullen theorem for $\hbu$.
\begin{theorem}
Let $\u$ be a holomorphy type with constants as in (\ref{constantes}). Let  $(X,q)$ be a Riemann domain spread over a Banach space $E$. Consider the following conditions.
\begin{itemize}
 \item[i)] $X$ is a $\hbu$-convex and $d_X(\hat A_{\hbu(X)})=d_X(A)$ for each open $X$-bounded set $A$.

\item[ii)] $X$ is a $\hbu$-convex.

\item[iii)] For each sequence $(x_n)\subset X$ such that $d_X(x_n)\to 0$, there exist a function $f\in\hbu(X)$ such that $\sup_n|f(x_n)|=\infty$.

\item[iv)] $X$ is an $\hbu$-domain of holomorphy.

\item[v)] For each open subset $A$ of $X$ which is not $X$-bounded there exist a function $f\in\hbu(X)$ such that $p_A(f):=\sup\{p_s^x(f):\, B_s(x)\textrm{ contained in }A \textrm{ and } X\textrm{-bounded}\}=\infty$.

\item[vi)] $X$ is an $\hbu$-$\hu$-domain of holomorphy.

\end{itemize}
Then $\mathrm{i)}\Rightarrow \mathrm{ii)}\Leftrightarrow \mathrm{iii)}\Leftrightarrow \mathrm{iv)}\Rightarrow\mathrm{v)}\Rightarrow\mathrm{vi)}$.

Moreover, if $\u$ is also weakly differentiable with constants as in (\ref{constantes}), then all the above conditions are equivalent.
\end{theorem}
\begin{proof}
The implications i) $\Rightarrow$ ii),  iii) $\Rightarrow$ v) and  iv) $\Rightarrow$ vi) are clear. The equivalence iii) and iv) is contained in \cite[Theorem 4.9]{Coe74}. 

Let us prove equivalence between ii) and iii). Suppose that $d_X(x_n)\to0$ and that $\tau(f)=\sup_n|f(x_n)|<\infty$ for every $f\in\hbu(X)$. Then the set $V=\{f\in\hbu(X):\, \tau(f)\le 1\}$ is absolutely convex and absorvent. Moreover $V$ is closed since it is the intersection of the sets $\{f\in\hbu(X):\, |f(x_n)|\le 1\}$ which are closed because evaluations at $x_n$ are continuous in $\hbu(X)$. 
Since $\hbu(X)$ is a barreled space, $V$ is a neighbourhood of 0 and thus $\tau$ is a continuous seminorm. Therefore, there are an $X$-bounded set $C$ and a constant $c>0$ such that $\tau(f)\le cp_C(f)$ for every $f\in\hbu(X)$. That is, $(x_n)\subset\hat C_{\hbu(X)}$, and thus $X$ is not $\hbu$-convex. Conversely, if $A$ is an $X$-bounded set such that $\hat A_{\hbu(X)}$ is not $X$-bounded, then there is a sequence $(x_n)$ in $\hat A_{\hbu(X)}$ such that $d_X(x_n)\to0$. This sequence satisfies that $\sup_n|f(x_n)|<\infty$.

% ii)$\Rightarrow$v):\newline Suppose that $A$ is an open subset of $X$ which is not $X$-bounded and $p_A(f)<\infty$ for every $f\in\hbu(X)$. Then the set $V=\{f\in\hbu(X):\, p_A(f)\le 1\}$ is absolutely convex and absorvent. Let us show that $V$ is closed. Let $A_n$ be the interior of $\{x\in A:\; d_A(x)\ge\frac1{n}\textrm{ and }\|q(x)\|\le n\}$. Then the sets $V(n)=\{f\in\hbu(X):\, p_{A_n}(f)\le 1\}$ are closed because $p_{A_n}$ are continuous seminorms. We prove that $V=\bigcap_nV(n)$. It suffices to show that $\sup_n p_{A_n}(f)=\lim_n p_{A_n}(f)\ge p_A(f)$ for every $f\in \hbu(X)$ (the other inequality is obvious). Let $B_s(x_0)\subset A$, then $B_{s-1/n}(x_0)\subset A_n$ for sufficiently large $n$. Moreover, since $\sum_ks^k\big\|\frac{d^kf(x_0)}{k!}\big\|_{\u_k(E)}<\infty$ (because $p_A(f)<\infty$), Abel's limit theorem implies that $\sum_ks^k\big\|\frac{d^kf(x_0)}{k!}\big\|_{\u_k(E)}=\lim_n\sum_k(s-1/n)^k\big\|\frac{d^kf(x_0)}{k!}\big\|_{\u_k(E)}$. Then $\lim_n p_{A_n}(f)\ge \lim_n p_{s-1/n}^{x_0}(f)=p_s^{x_0}(f)$. Therefore, $\lim_n p_{A_n}(f)\ge p_A(f)$.
% 
% 
% Since $\hbu(X)$ is a barreled space, $V$ is a neighbourhood of 0 and thus $p_A$ is continuous. Therefore, there are an $X$-bounded set $C$ and a constant $c>0$ such that $p_A(f)\le cp_C(f)$ for every $f\in\hbu(X)$. That is, $A\subset\hat C_{\hbu(X)}$, and thus $X$ is not $\hbu$-convex.

We prove now that v) implies vi). Let $\tau:X\to Y$ be a morphism which is an $\hbu$-$\hu$-extension but is not surjective and take $y$ in the border of $X$. Let $(x_n)$ be a sequence contained in $X$ converging to $y$ and let $B_n$ be the ball of center $x_n$ and radius $\frac{d_X(x_n)}{2}$. Since $A=\cup_nB_n$ is not $X$-bounded, there is some $f\in\hbu(X)$ such that $p_A(f)=\infty$. Let $A_k=\cup_{n\ge k}B_n$, then clearly $p_{A_k}(f)=\infty$ for every $k\ge1$.
Since $f$ extends to $\tilde f\in\hu(Y)$, there is some $r>0$ such that $p_{B_r(y)}(\tilde f)<\infty$. Moreover, if $k$ is large enough then $A_k\subset B_r(y)$. Thus, for $k$ large enough, we have that $p_{A_k}(f)=p_{A_k}(\tilde f)\le p_{B_r(y)}(\tilde f)<\infty$, which is a contradiction.

It remains to prove that vi) implies i) when $\u$ is weakly differentiable.\newline
\textit{Claim:} If $\u$ be a weakly differentiable holomorphy type with constants as in (\ref{constantes}), $A$ an open $X$-bounded set and $y\in\hat A_{\hbu(X)}$ and $f\in\hbu(X)$, then $f\circ(q|_{B_y})^{-1}$ extends to a function $\tilde f\in\hbu(B_{d_X(A)}(q(y)))$.\newline
%
% Moreover, for $\alpha<d_X(A)$, $p_\alpha^{q(y)}(f)\le \frac{c_A}{d_X(A)-\alpha} p_{\tilde A}(f)$, where $\tilde A=\bigcup_{x\in A}B(x,\frac{2d_X(A)+\alpha}{3})$.
\textit{Proof of the claim.} Let $\alpha<\alpha_0<d_X(A)$. Then by Lemma \ref{lema1 cartan-thullen}, there exists a constant $C$ (independent of $k$) such that $\alpha_0^k\sup_{\varphi_k\in B_{\u_k(E)'}}p_A(\varphi_k\circ\frac{d^kf}{k!}) \le C p_{\tilde A}(f)$. Thus,
\begin{eqnarray*}
\sum_k\alpha^k\big\|\frac{d^kf}{k!}(y)\big\|_{\u_k(E)} & = & \sum_k\alpha^k\sup_{\varphi_k\in B_{\u_k(E)'}}\big|\varphi_k(\frac{d^kf}{k!}(y))\big|\le \sum_k\alpha^k\sup_{\varphi_k\in B_{\u_k(E)'}}cp_A(\varphi_k\circ\frac{d^kf}{k!}) \\
& \le & c \sum_k\big(\frac{\alpha}{\alpha_0}\big)^kCp_{\tilde A}(f) = cC\frac{\alpha_0}{\alpha_0-\alpha} p_{\tilde A}(f) <\infty.
\end{eqnarray*}
Since this is true for every $\alpha<d_X(A)$, we have that the Taylor series of $f\circ(q|_{B_y})^{-1}$ at $q(y)$ converges on $B_{d_X(A)}(q(y))$ and that $f\circ(q|_{B_y})^{-1}$ belongs to $\hbu(B_{d_X(A)}(q(y)))$, %$p_\alpha^{q(y)}(f)\le \frac{c_A}{d_X(A)-\alpha} p_{\tilde A}(f)$.
 and the claim is proved.

Suppose now that  $d_X(\hat A_{\hbu(X)})<d_X(A)$. Take $y\in\hat A_{\hbu(X)}$ such that $d_X(y)<d_X(A)$. We define a Riemann domain $\tilde X$ adjoining to $X$ the ball $B_{d_X(A)}(q(y))$ as follows. First define a Riemann domain $(X_0,q_0)$ as the disjoint union $X\cup B_{d_X(A)}(q(y))$, and $q_0(x_0)=q(x_0)$ if $x_0\in X$; $q_0(x_0)=x_0$ if $x_0\in B_{d_X(A)}(q(y))$. Then define the following equivalence relation $\sim$ on $X_0$: each point is related with itself and two points $x_0\in B_{d_X(A)}(q(y))$, $x_1\in X$ are related if and only if $q(x_1)=x_0$ and $x_1$ may be joined to $y$ by a curve contained in $q^{-1}(B_{d_X(A)}(q(y)))$. Let $\tilde X$ be the Riemann domain $X_0/_\sim$. By the claim the inclusion $X\hookrightarrow\tilde X$ is an $\hbu$-$\hu$-extension morphism, and it is not an isomorphism.
\end{proof}
\begin{remark}\rm
By \cite[Theorem 4.9]{Coe74} we also have that if $\u$ is a holomorphy type with constants as in (\ref{constantes}) and $(X,q)$ is a domain over a separable Banach space $E$, then $X$ is an $\hbu$-domain of holomorphy if and only if $X$ is the domain of existence of a function $f\in\hbu(X)$.
\end{remark}

% \textcolor{red}{pregunta:} si $\varphi\in M_b$   $\varphi\prec B_1\cup B_2$, y las bolas son disjuntas, vale que $\varphi\prec B_1$ o $\varphi\prec B_2$?

\section*{Acknowledgements} 
It is a pleasure to thank Daniel Carando and Ver\'onica Dimant for many helpful conversations. Serveral parts of this article have been considerably improved by their comments. I would also like to thank Pablo Turco for pointing me out some useful results in \cite{Nac69}.

% \bibliography{biblio}
% \bibliographystyle{plain}

\end{document}